\documentclass[11pt,headings=standardclasses,DIV=13]{scrartcl}
\usepackage[utf8]{inputenc}
\usepackage[english]{babel}
\usepackage{amsmath}
\usepackage{amsfonts,amssymb,mathrsfs,textcomp,amsthm}
\usepackage[mathscr]{eucal}
\usepackage{tikz}
\usepackage{tikz-3dplot}
\usepackage{pgfplots}
\usepackage{subfig}
\usepackage{hyperref}
\usepackage[all]{hypcap}
\usepackage[capitalize,nameinlink]{cleveref}
\usepackage[pagewise]{lineno}%\linenumbers 

\crefname{section}{Section}{Sections}
\crefname{subsection}{Subsection}{Subsections}
\Crefname{section}{Section}{Sections}
\Crefname{subsection}{Subsection}{Subsections}
\Crefname{figure}{Figure}{Figures}
\crefformat{equation}{\textup{#2(#1)#3}}
\crefrangeformat{equation}{\textup{#3(#1)#4--#5(#2)#6}}
\crefmultiformat{equation}{\textup{#2(#1)#3}}{ and \textup{#2(#1)#3}}
{, \textup{#2(#1)#3}}{, and \textup{#2(#1)#3}}
\crefrangemultiformat{equation}{\textup{#3(#1)#4--#5(#2)#6}}%
{ and \textup{#3(#1)#4--#5(#2)#6}}{, \textup{#3(#1)#4--#5(#2)#6}}{, and \textup{#3(#1)#4--#5(#2)#6}}
\Crefformat{equation}{#2Equation~\textup{(#1)}#3}
\Crefrangeformat{equation}{Equations~\textup{#3(#1)#4--#5(#2)#6}}
\Crefmultiformat{equation}{Equations~\textup{#2(#1)#3}}{ and \textup{#2(#1)#3}}
{, \textup{#2(#1)#3}}{, and \textup{#2(#1)#3}}
\Crefrangemultiformat{equation}{Equations~\textup{#3(#1)#4--#5(#2)#6}}%
{ and \textup{#3(#1)#4--#5(#2)#6}}{, \textup{#3(#1)#4--#5(#2)#6}}{, and \textup{#3(#1)#4--#5(#2)#6}}
\crefdefaultlabelformat{#2\textup{#1}#3}

\usetikzlibrary{math} 
\usepgfplotslibrary{fillbetween}
\usetikzlibrary{quotes,angles,babel}
\hypersetup{
    bookmarksnumbered=true,
    unicode=false,
    pdfmenubar=true,
    pdffitwindow=false,
    pdfstartview={FitH},
    pdftitle={Analysis of directional higher order jump discontinuities with trigonometric shearlets},
    pdfauthor={Kevin Schober},
    pdfcreator={Kevin Schober},
    pdfnewwindow=true,
	colorlinks=true,
	linkcolor=blue,
	citecolor=blue,
	filecolor=blue,
	urlcolor=blue
} 

\newtheorem{theorem}{Theorem}[section]
\newtheorem{lemma}{Lemma}[section]

\theoremstyle{definition}
\newtheorem{definition}{Definition}[section]
\newtheorem{remark}{Remark}[section]

\newcommand{\norm}[1]{\left\lVert#1\right\rVert}
\newcommand{\abs}[1]{\left\lvert#1\right\rvert}
\makeatletter
\renewenvironment{thebibliography}[1]
     {\section*{\refname}%
      \@mkboth{\MakeUppercase\refname}{\MakeUppercase\refname}%
      \list{\@biblabel{\@arabic\c@enumiv}}%
           {\settowidth\labelwidth{\@biblabel{#1}}%
            \leftmargin2em
            \itemindent0em
			\itemsep0.3em
            \@openbib@code
            \usecounter{enumiv}%
            \let\p@enumiv\@empty
            \renewcommand\theenumiv{\@arabic\c@enumiv}}%
      \sloppy
      \clubpenalty4000
      \@clubpenalty \clubpenalty
      \widowpenalty4000%
      \sfcode`\.\@m}
     {\def\@noitemerr
       {\@latex@warning{Empty `thebibliography' environment}}%
      \endlist}
\makeatother

\makeatletter
\let\@fnsymbol\@arabic
\makeatother

\title{Analysis of directional higher order jump discontinuities with trigonometric shearlets}
\date{}
\author{Kevin Schober\thanks{Corresponding author} \textsuperscript{,}\thanks{Institute of Mathematics, University of L\"ubeck, Ratzeburger Allee 160, D-23562 L\"ubeck, Germany.\newline
E-mail: \href{mailto:schober@math.uni-luebeck.de}{schober@math.uni-luebeck.de} (K. Schober), \href{mailto:prestin@math.uni-luebeck.de}{prestin@math.uni-luebeck.de} (J. Prestin)}
 \and J\"urgen Prestin\footnotemark[2]}

\begin{document}

\maketitle

% Abstract
\begin{abstract}  
In a recent article, we showed that trigonometric shearlets are able to detect directional step discontinuities along edges of periodic characteristic functions. In this paper, we extend these results to multivariate periodic functions which have jump discontinuities in higher order directional derivatives along edges. In order to prove suitable upper and lower bounds for the shearlet coefficients, we need to generalize the results about localization- and orientation-dependent decay properties of the corresponding inner products of trigonometric shearlets and the underlying periodic functions. 
\end{abstract}



\smallskip

% Section Introduction
\section{Introduction} 
\label{sec:introduction} 

The automatic recognition and separation of different image parts is of great importance in many industrial or life science applications. For this reason, one needs to precisely and effectively detect edges in images. One famous approach is the Canny algorithm \cite{canny:detection} which applies two-dimensional edge filters on a smoothed version of the image followed by a non-maximum suppression called hysteresis. It is well-known that the Canny algorithm is equivalent to the task of finding local maxima of a two-dimensional wavelet transform \cite{mallat:detection}. Typically, classical multivariate wavelets are obtained by taking the tensor product of one-dimensional scaling and wavelet functions. Since the support of these functions is aligned with the coordinate axes, they are not optimal for the detection and characterization of singularities in arbitrary directions as they can occur in dimensions higher than one \cite{mallat:buch}. Therefore, several multivariate directional systems have been considered in order to overcome this limitation, for example brushlets \cite{meyer:brushlets}, ridgelets \cite{candes:ridgelets}, curvelets \cite{candes:curvelets} or shearlets \cite{kutyniok:book}.

A widely used model for multivariate functions which contain singularities along edges is the class of so-called cartoon-like functions \cite{donoho:wedgelets,donoho:sparse}. These are functions of the form $\mathfrak{f}=f_0+f_1\,\chi_T$ where $T\subset \mathbb{R}^2$ and $f_0,f_1$ are smooth functions with compact support. This class was used for optimal sparse approximation with multivariate directional systems such as curvelets or shearlets \cite{candes:curvelets,labate:sparse} and later for the more general classes of parabolic molecules or $\alpha$-molecules \cite{grohs:parabolic_molecules,grohs:alpha, grohs:molecules}. Another application is the detection and characterization of directional discontinuities in cartoon-like functions. In a number of articles \cite{grohs:parabolic_molecules,labate:detection_continuous,kutyniok:edges_compactly} it was shown for different settings that continuous shearlets are well suited to deal with this task if the underlying cartoon-like function is piecewise constant.

To get a more realistic model of images with smooth transitions of different image parts, one needs to allow for the functions $f_1$ to be smooth with vanishing values on the boundary curve up to a directional derivative of higher order. In \cite{labate:smooth}, the authors showed that for the continuous shearlet coefficients the estimate
\begin{equation}\label{eq:main_result_cont}
	0<\lim\limits_{a\rightarrow 0^+}a^{-(n/2+3/4)}\abs{\bigl\langle \mathfrak{f},\psi_{a,s_0,\mathbf{p}} \bigr\rangle}<\infty
\end{equation} 
holds true if $\mathbf{p}\in\partial T$ and $s=s_0$ corresponds to the normal direction of $\partial T$ at $\mathbf{p}$ with $n$ denoting the number of vanishing derivatives of $f_1$ in that point. On the other hand, the shearlet coefficients exhibit rapid decay if $\mathbf{p}\notin\partial T$ or if $s = s_0$ does not correspond to the normal direction of $\partial T$ at $\mathbf{p}$.

In the case of discrete shearlets, the authors in \cite{labate:detection} proved the existence of suitable upper and lower bounds for the shearlet coefficients if the corresponding cartoon-like function is piecewise constant, e.g. $f_0=0$ and $f_1=1$. In \cite{schober:detection}, a similar result was shown for trigonometric shearlets and the detection of singularities of periodic characteristic functions. Until now, there is no analogous result to \cref{eq:main_result_cont} for the detection of jumps in higher order directional derivatives in the discrete setting. 

In this paper, we consider the trigonometric shearlets from \cite{schober:detection} which arise from the theory of multivariate periodic wavelets \cite{bergmann:dlVP,langemann:multi_periodic_wavelets,maksimenko:multi_periodic_wavelets} and extend the results to general cartoon-like functions having jumps in higher order directional derivatives on edge curves which hence do not need to be closed as in the case of characteristic functions in \cite{labate:detection, schober:detection}. We provide upper and lower estimates for the shearlet coefficients in the case that the corresponding smooth function $f_1$ vanishes on the boundary curve up to a directional derivative of higher order. 

The structure of the paper is as follows: We introduce trigonometric shearlets in \cref{sec:trigonometric_shearlets} and show a new upper bound for the partial derivates of these functions in polar coordinates. In \cref{sec:main_results}, we formulate the two main results of this paper given by \cref{thm:hauptresultat} und \cref{thm:hauptresultat2}. The next \cref{sec:proof_of_theorem_3_1} contains the proof of \cref{thm:hauptresultat} based on a decomposition of the underlying cartoon-like function on dyadic squares. \cref{sec:localization_lemmata} includes technical preparations for the proof of \cref{thm:hauptresultat2} in terms of localization lemmata and a new representation of the Fourier transform of polynomial cartoon-like functions in \cref{lem:fourier_transformation_gauss}. With these results in hand, we give proof of the lower bound for the shearlet coefficients in \cref{sec:proof_of_theorem_3_2}.

% Section: Trigonometric shearlets
 \section{Trigonometric shearlets}
 \label{sec:trigonometric_shearlets}

We denote two-dimensional vectors by $\mathbf{x}=(x_1,x_2)^{\mathrm{T}}$ with the usual inner product $\mathbf{x}^{\mathrm{T}}\mathbf{y}\mathrel{\mathop:}= x_1\,y_1+x_2\,y_2$ and the induced Euclidean norm written as $\abs{\mathbf{x}}_2\mathrel{\mathop:}=\sqrt{\mathbf{x}^{\mathrm{T}}\mathbf{x}}$. Moreover, we write $\abs{\mathbf{x}}_1\mathrel{\mathop:}=\abs{x_1}+\abs{x_2}$, $\mathbf{x}^\mathbf{y}\mathrel{\mathop:}=x_1^{y_1}\,x_2^{y_2}$ and $\mathbf{x}^\beta\mathrel{\mathop:}=x_1^\beta\,x_2^\beta$ for $\beta\in \mathbb{R}$. For the representation of a vector $\boldsymbol{\xi}\in \mathbb{R}^2$ in polar coordinates, we write $\boldsymbol{\xi}=\rho\,\boldsymbol{\Theta}(\theta)$ with $\rho\mathrel{\mathop:}=\abs{\boldsymbol{\xi}}_2$ and $\boldsymbol{\Theta}(\theta)\mathrel{\mathop:}=(\cos\theta,\sin\theta)^{\mathrm{T}}$. For $\mathbf{k}\in \mathbb{N}_0^2$ and $n\in \mathbb{N}_0$ with $\abs{\mathbf{k}}_1\leq n$ we define $\mathbf{k}!\mathrel{\mathop:}=k_1!\,k_2!$ and $\binom{n}{\mathbf{k}}\mathrel{\mathop:}=\frac{n!}{\mathbf{k}!(n-\abs{\mathbf{k}}_1)!}$.  We denote by $C(\Omega)$ the space of all continuous functions on a domain $\Omega\subseteq \mathbb{R}^2$ with the norm $\norm{f}_{C(\Omega)}\mathrel{\mathop:}=\sup\limits_{\mathbf{x}\in A}\abs{f(\mathbf{x})}$. For $\mathbf{r}=(r_1,r_2)^{\mathrm{T}}\in \mathbb{N}_0^2$ and a sufficiently smooth function $f$ we use the notation $\partial^{\mathbf{r}}f\mathrel{\mathop:}= \frac{\partial^{r_1+r_2}}{\partial x_1^{r_1}\partial x_2^{r_2}}f$ and the space of all $q$-times continuously differentiable compactly supported functions will be denoted by
 \begin{equation*}
 	C^q_0(\Omega)\mathrel{\mathop:}=\left\lbrace f:\Omega\rightarrow \mathbb{R}:\partial^\mathbf{r}f\in C(\Omega)\;\text{for all}\;\mathbf{r}\in \mathbb{N}_0^2\; \text{with}\;\abs{\mathbf{r}}_1\leq q,\,\abs{\mathrm{supp}\,f}<\infty\right\rbrace
 \end{equation*}
 with the norm $\norm{f}_{C_0^q}\mathrel{\mathop:}=\norm{f}_{C_0^q(\Omega)}\mathrel{\mathop:}=\sup\limits_{\abs{\mathbf{r}}_1\leq q}\,\sup\limits_{\mathbf{x}\in \Omega}\abs{\partial^\mathbf{r}f(\mathbf{x})}$.\\

In this section, we define trigonometric shearlets which were already used in \cite{schober:detection}. For convenience, we briefly recap the construction and some properties of these functions. 
We call a nonnegative and even function $g:\mathbb{R}\rightarrow \mathbb{R}$ admissible if $\mathrm{supp}\,g=\left(-\frac{2}{3},\frac{2}{3}\right)$ and $g$ is monotonically decreasing for $x\in \left( \frac{1}{3},\frac{2}{3} \right)$ and satisfies the property $\sum\limits_{z\in \mathbb{Z}}g(x+z)=1$ for all $x\in \mathbb{R}$.\\
 An admissible function $g$ can be chosen arbitrarily smooth \cite{schober:detection}. We introduce functions $\widetilde{g}:\mathbb{R}\rightarrow \mathbb{R}$ which are given by $\widetilde{g}(x)\mathrel{\mathop:}=g\left( \frac{x}{2} \right)-g(x)$.\\

For $\mathfrak{i}\in\lbrace \mathfrak{h},\mathfrak{v}\rbrace$ we consider bivariate functions $\Psi^{(\mathfrak{i})}:\mathbb{R}^2\rightarrow \mathbb{R}$ defined by
 	\begin{equation}\label{eq:window_function}
 		\Psi^{(\mathfrak{h})}(\mathbf{x})\mathrel{\mathop:}=\widetilde{g}(x_1)\,g(x_2),\qquad\qquad \Psi^{(\mathfrak{v})}(\mathbf{x})\mathrel{\mathop:}=g(x_1)\,\widetilde{g}(x_2).
 	\end{equation}	
	We call them window functions and write $\Psi^{(\mathfrak{i})}\in \mathcal{W}$. We remark that for an admissible function $g\in C^q_0(\mathbb{R})$ we have $\Psi^{(\mathfrak{i})}\in C^q_0(\mathbb{R}^2)$ and denote $\Psi^{(\mathfrak{i})}\in \mathcal{W}^q$. For even $j\in \mathbb{N}_0$ and $\ell\in \mathbb{Z}$ with $\abs{\ell}\leq 2^{j/2}$ we consider the matrices
 		\begin{equation*}
 			\mathbf{N}_{j,\ell}^{(\mathfrak{h})}\mathrel{\mathop:}=\begin{pmatrix}
 				2^j & \ell\, 2^{j/2}\\
 				0 & 2^{j/2}
 			\end{pmatrix},\qquad\qquad\;\mathbf{N}_{j,\ell}^{(\mathfrak{v})}\mathrel{\mathop:}=\begin{pmatrix}
 				2^{j/2} & 0\\
 				\ell\, 2^{j/2} & 2^j
 			\end{pmatrix}
 		\end{equation*}
 		and introduce the functions
 	\begin{equation*}
\Psi^{(\mathfrak{i})}_{j,\ell}(\mathbf{x})\mathrel{\mathop:}=\Psi^{(\mathfrak{i})}\left(\left(\mathbf{N}_{j,\ell}^{(\mathfrak{i})}\right)^{-\mathrm{T}}\mathbf{x}\right),\qquad \mathbf{x}\in \mathbb{R}^2.
 	\end{equation*}  
 For $\mathfrak{i}\in\lbrace \mathfrak{i},\mathfrak{v}\rbrace$, $\Psi^{(\mathfrak{i})}\in \mathcal{W}^q$ and $\mathbf{y}\in \mathcal{P}(\mathbf{N}_{j,\ell}^{(\mathfrak{i})})$ we define trigonometric shearlets by
 	\begin{equation}\label{eq:trigonometric_shearlets}
 		\psi_{j,\ell,\mathbf{y}}^{(\mathfrak{i})}(\mathbf{x})\mathrel{\mathop:}=2^{-3j/4}\sum_{\mathbf{k}\in \mathbb{Z}^2}\Psi^{(\mathfrak{i})}_{j,\ell}(\mathbf{k})\,\mathrm{e}^{\mathrm{i}\mathbf{k}^{\mathrm{T}}(\mathbf{x}-2\pi\mathbf{y})},
 	\end{equation}
	where 
 	\begin{align*}
 		&\mathcal{P}\left(\mathbf{N}_{j,\ell}^{(\mathfrak{h})}\right)=\Bigl\lbrace 2^{-j}\,z_1\,:\,z_1=-2^{j-1},\dots,2^{j-1}-1 \Bigr\rbrace\times\Bigl\lbrace 2^{-j/2}\,z_2\,:\,z_2=-2^{j/2-1},\dots,2^{j/2-1}-1\Bigr\rbrace,\\
 		&\mathcal{P}\left(\mathbf{N}_{j,\ell}^{(\mathfrak{v})}\right)=\Bigl\lbrace 2^{-j/2}\,z_1\,:\,z_1=-2^{j/2-1},\dots,2^{j/2-1}-1 \Bigr\rbrace\times\Bigl\lbrace 2^{-j}\,z_2\,:\,z_2=-2^{j-1},\dots,2^{j-1}-1 \Bigr\rbrace.
 	\end{align*}
	
Let $f,g:\mathbb{R}\rightarrow \mathbb{R}$ be sufficiently smooth functions. The $n$-th order derivative of the composition of $f$ and $g$ is given by Fa\`{a} di Bruno's formula \cite[Section 4.3]{porteous:faadibruno}
	    \begin{align}
	    	\frac{\mathrm{d}^n}{\mathrm{d}x^n}f(g(x))&=\sum\limits_{\mathbf{k}}\binom{n}{\mathbf{k}}\,f^{(\abs{\mathbf{k}}_1)}(g(x))\,\prod\limits_{j=1}^n \left( \frac{g^{(j)}(x)}{j!} \right)^{k_j}\notag\\\label{eq:faa_di_bruno2}
	    	&=\sum\limits_{k=1}^{n}f^{(k)}(g(x))\,B_{n,k}\Bigl(g'(x),g''(x),\hdots,g^{(n-k+1)}(x)\Bigr),
	    \end{align}
	    where the sum in the first line runs over all vectors $\mathbf{k}=(k_1,\hdots,k_n)^{\mathrm{T}}\in \mathbb{N}_0^n$ with $\sum\limits_{i=1}^n i\cdot k_i=n$ and $B_{n,k}$ are the well known Bell polynomials. It is also known that
	    \begin{equation}\label{eq:bell_zahl}
	    	\sum_{k=0}^{n}B_{n,k}(1,1,\hdots,1)=\sum\limits_{k=0}^{n}\sum\limits_{\mathbf{m}}\binom{n}{\mathbf{m}}\prod_{j=1}^{n-k+1}(j!)^{-m_j}=B_n,
	    \end{equation}
	    where $B_n$ is the $n$-th Bell number and the inner sum is running over all $\mathbf{m}=(m_1,\hdots,m_{n-k+1})^{\mathrm{T}}\in \mathbb{N}_0^{n-k+1}$ fulfilling $\sum\limits_{i=1}^{n-k+1}m_i=k$ and $\sum\limits_{i=1}^n i\cdot m_i=n$. In the following lemma we need the angles
		\begin{equation}\label{eq:theta_jl}
			\theta_{j,\ell}^{(\mathrm{h})}\mathrel{\mathop:}=\arctan\left(\ell\,2^{-j/2}\right),\qquad\qquad \theta_{j,\ell}^{(\mathrm{v})}\mathrel{\mathop:}=\mathrm{arccot}\left(\ell\,2^{-j/2}\right).
		\end{equation}
 \begin{lemma}\label{lem:partielle_ableitung_psi:polar}
 	For $\mathfrak{i}\in\lbrace \mathfrak{h},\mathfrak{v}\rbrace$ and $q\in \mathbb{N}_0$ let $\Psi^{(\mathfrak{i})}\in \mathcal{W}^q$ from \cref{eq:window_function} be given. Then for all $r\leq q$ we have
 	\begin{equation*}
 		\abs{\frac{\partial^r}{\partial\rho^r}\left[\Psi_{j,\ell}^{(\mathfrak{i})}\left( 2^j\rho\,\boldsymbol{\Theta}(\theta) \right)\right]}\leq C_1(r),\qquad\abs{\frac{\partial^r}{\partial\theta^r}\left[\Psi_{j,\ell}^{(\mathfrak{i})}\left( 2^j\rho\,\boldsymbol{\Theta}(\theta) \right)\right]}\leq C_2(r)\,2^{jr/2}.
 	\end{equation*}
 \end{lemma}
 \begin{proof}
 	We only show the case $\mathfrak{i}=\mathfrak{h}$. We use polar coordinates and obtain
 	\begin{equation*}
 		\Psi_{j,\ell}^{(\mathfrak{h})}\left( 2^j\rho\,\boldsymbol{\Theta}(\theta) \right)=\widetilde{g}_{\alpha}\bigl(\rho\cos\theta\bigr)\,g_{\alpha}\Bigl( \rho\cos\theta\left(2^{j/2}\tan\theta-\ell\right) \Bigr)
 	\end{equation*}
 	and use the chain rule to get
 	\begin{equation}\label{eq:partielle_ableitung_psi:polar1}
		\abs{\frac{\partial^s}{\partial\rho^s}\left[\widetilde{g}_{\alpha}\bigl(\rho\cos\theta\bigr)\right]}=\abs{\cos\theta}^s\,\abs{\widetilde{g}_{\alpha}^{(s)}\bigl(\rho\cos\theta\bigr)}\leq \norm{\widetilde{g}_{\alpha}}_{C^s(\mathbb{R})}=C_1(s)
 	\end{equation}
	for all $s\leq r$. From \cite[Lemma 1]{schober:detection} it follows that
 	  \begin{equation*}
 	  	\mathrm{supp}\,\Psi_{j,\ell}^{(\mathfrak{h})}\left( 2^j\rho\,\boldsymbol{\Theta}(\theta) \right)\subset\left\lbrace(\rho,\theta)\in \mathbb{R}\times\left[-\frac{\pi}{2},\frac{\pi}{2}\right]:\frac{1}{3}<\abs{\rho}< 2,\,\theta_{j,\ell-2}^{(\mathfrak{h})}<\theta<\theta_{j,\ell+2}^{(\mathfrak{h})}\right\rbrace,
 	  \end{equation*} 
 	 leading to the estimate
 	 \begin{align}
 	 	\abs{\frac{\partial^s}{\partial\rho^s}\left[g_{\alpha}\Bigl( \rho\cos\theta\left(2^{j/2}\tan\theta-\ell\right) \Bigr)\right]}&=\label{eq:partielle_ableitung_psi:polar2} \abs{\cos\theta}^s\Bigl\lvert2^{j/2}\tan\theta-\ell \Bigr\rvert^s \abs{g_{\alpha}^{(s)}\Bigl( \rho\cos\theta\left(2^{j/2}\tan\theta-\ell\right) \Bigr)}\notag\\
 		&\leq 2^s\,\norm{g_{\alpha}}_{C^s(\mathbb{R})}\leq C_2.
 	 \end{align}
 	 Using Leibniz rule and triangle inequality we get with \cref{eq:partielle_ableitung_psi:polar1} and \cref{eq:partielle_ableitung_psi:polar2}
 	 \begin{equation*}
 	 	\abs{\frac{\partial^r}{\partial\rho^r}\left[\Psi_{j,\ell}^{(\mathfrak{h})}\left( 2^j\rho\,\boldsymbol{\Theta}(\theta) \right)\right]}\leq\sum\limits_{s=0}^{r}\binom{r}{s}\,2^{r-s}\,C_1(s)\,C_2(r-s)\leq 3^r\,C_3(r)=C_4(r).
 	 \end{equation*}
 	 For the variable $\theta$ we again use the chain rule for $s\leq r$ to obtain
 	 \begin{equation*}
 	 	\abs{\frac{\partial^s}{\partial\theta^s}\left[ \rho\cos\theta\right]}\leq\abs{\rho},
 	 \end{equation*}
 	 which leads with \cref{eq:bell_zahl} and the Fa\`{a} di Bruno formula from \cref{eq:faa_di_bruno2} to
 	 \begin{align}
 	 	\abs{\frac{\partial^s}{\partial\theta^s}\left[ \widetilde{g}_{\alpha}\bigl(\rho\cos\theta\bigr)\right]}&\leq\sum\limits_{t=1}^{s}\abs{\widetilde{g}_{\alpha}^{(t)}\bigl(\rho\cos\theta\bigr)}\,B_{s,t}\left(\abs{\frac{\partial}{\partial\theta}\left[ \rho\cos\theta\right]},\hdots,\abs{\frac{\partial^{s-t+1}}{\partial\theta^{s-t+1}}\left[ \rho\cos\theta\right]}\right)\notag\\
		&\leq\norm{\widetilde{g}_{\alpha}}_{C^s(\mathbb{R})}\sum\limits_{t=1}^{s}\sum\limits_{\mathbf{m}}\binom{s}{\mathbf{m}}\prod_{j=1}^{s-t+1}\left(\frac{\abs{\frac{\partial^j}{\partial\theta^j}\left[ \rho\cos\theta\right]}}{j!}\right)^{m_j}\notag\\
 		&\leq C_5(s)\sum\limits_{t=1}^{s}\abs{\rho}^{t}\sum\limits_{\mathbf{m}}\binom{s}{\mathbf{m}}\,\prod\limits_{j=1}^{s-t+1}(j!)^{-m_j}\leq C_6(s)\,B_s=C_7(s).
\label{eq:partielle_ableitung_psi:polar3}
 		 \end{align} 
 		 For even $s\in \mathbb{N}$ we have
 		 \begin{equation*}
 		 	\abs{\frac{\partial^s}{\partial\theta^s}\left[\rho\cos\theta\left(2^{j/2}\tan\theta-\ell\right)\right]}= \abs{\rho}\abs{\cos\theta}\Bigl\lvert2^{j/2}\tan\theta-\ell \Bigr\rvert \leq 4
 		 \end{equation*}
		 since $\rho\leq 2$ and for odd $s\in \mathbb{N}$ we see
 		 \begin{equation*}
 		 	\abs{\frac{\partial^s}{\partial\theta^s}\left[\rho\cos\theta\left(2^{j/2}\tan\theta-\ell\right)\right]}= \abs{\rho}\abs{\cos\theta}\Bigl\lvert2^{j/2}+\ell\,\tan\theta \Bigr\rvert \leq C_8\,\abs{\rho}\,2^{j/2}
 		 \end{equation*}
 		 since $\abs{\ell}<2^{j/2}$. Using the Fa\`{a} di Bruno formula \cref{eq:faa_di_bruno2} we obtain the estimate
 		 \begin{align}
 		 	&\abs{\frac{\partial^s}{\partial\theta^s}\left[g_{\alpha}\Bigl( \rho\cos\theta\left(2^{j/2}\tan\theta-\ell\right) \Bigr)\right]}\notag\\ 
			&\quad\leq\norm{g_{\alpha}}_{C^s(\mathbb{R})}\sum\limits_{t=1}^{s}\sum\limits_{\mathbf{m}}\binom{s}{\mathbf{m}}\prod_{j=1}^{s-t+1}\left(\frac{\abs{\frac{\partial^j}{\partial\theta^j}\left[ \rho\cos\theta\left(2^{k}\tan\theta-\ell\right)\right]}}{j!}\right)^{m_j}\notag\\
 			&\quad\leq C_9(s)\sum\limits_{t=1}^{s}\abs{\rho}^{t}\,2^{jt/2}\sum\limits_{\mathbf{m}}\binom{s}{\mathbf{m}}\,\prod\limits_{j=1}^{s-t+1}(j!)^{-m_j}\leq C_{10}(s)\,2^{js/2}\,B_s=C_{11}(s)\,2^{js/2}.
\label{eq:partielle_ableitung_psi:polar4}
 		 \end{align}
 		 With the estimates \cref{eq:partielle_ableitung_psi:polar3} and \cref{eq:partielle_ableitung_psi:polar4} we finally conclude
 		 \begin{align*}
 		 	\abs{\frac{\partial^r}{\partial\theta^r}\left[\Psi_{j,\ell}^{(\mathfrak{h})}\left( 2^j\rho\,\boldsymbol{\Theta}(\theta) \right)\right]}&\leq\sum\limits_{s=0}^{r}\binom{r}{s}\,\abs{\frac{\partial^s}{\partial\theta^s}\left[\widetilde{g}_{\alpha}\bigl(\rho\cos\theta\bigr)\right]}\\
 			&\qquad\times\abs{\frac{\partial^{r-s}}{\partial\theta^{r-s}}\left[g_{\alpha}\Bigl( \rho\cos\theta\left(2^{k}\tan\theta-\ell\right) \Bigr)\right]}\\
 			&\leq\sum\limits_{s=0}^{r}\binom{r}{s}\,C_7(s)\,B_s\,C_{11}(r-s)\,2^{j(r-s)/2}\,B_{r-s}\\
 			&\leq C_{12}(r)\,B_r^2\,2^{jr/2}=C_{13}(r)\,2^{jr/2}.
 		 \end{align*}
 \end{proof}
 
% Subsection: Trigonometric shearlets (end)

% Section introduction (end)

 % Section: Main results
 \section{Main results}
 \label{sec:main_results}
 
 For the main results of this paper, we need the class of so called cartoon-like functions \cite{candes:curvelets,donoho:wedgelets,labate:sparse}. These are functions which are smooth except for discontinuities along edges. We call a set $T\subset \left( -\pi,\pi \right)^2$ star-shaped and write $T\in \mathrm{STAR}$ if there exists $\mathbf{x}_0\in T$, called origin, such that for every $\mathbf{x}\in T$ we have
\begin{equation*}
	\bigl\lbrace \lambda \mathbf{x}+(1-\lambda)\mathbf{x}_0\,:\,\lambda\in[0,1] \bigr\rbrace\subset T.
\end{equation*}
We follow the ideas of \cite[Section 8.2]{donoho:wedgelets} and consider star-shaped sets with smooth boundaries $\partial T$ given by a parametrized curve in polar coordinates. Let $r\in C^2 \left([0,2\pi)\right)$ be a radius function with $\norm{r}_{C^2}\leq\tau$ and $T\in \mathrm{STAR}$ a star-shaped set with origin $\mathbf{x}_0$ which boundary $\partial T$ can be expressed in polar coordinates by a parametrized curve $\boldsymbol{\gamma}:[0,2\pi)\rightarrow \partial T$ of the form
 	\begin{equation}
 		\label{eq:star1}
 		\boldsymbol{\gamma}(x)=\mathbf{x}_0+r(x)\,(\cos x,\,\sin x)^{\mathrm{T}},\qquad x\in[0,2\pi).
 		\end{equation}
 		The set $\mathrm{STAR}^2(\tau)$ is defined as the set which contains all $T\in \mathrm{STAR}$ with a boundary described as in \cref{eq:star1}. 
 % Definition: Cartoon like Funktionen
 \begin{definition}\label{def:cartoon_funktionen}
 	For $T\in \mathrm{STAR}^2(\tau)$ and $u\in \mathbb{N}_0$ the set of cartoon-like functions is defined by
 	\begin{equation*}
 		\mathcal{E}^u(\tau)\mathrel{\mathop:}=\Bigl\lbrace \mathfrak{f}=f_0+f_1\chi_T\,:\,f_0,f_1\in C_0^u(\mathbb{R}^2)\;\,\text{and}\;\, \mathrm{supp}\,f_0\subset(-\pi,\pi)^2\Bigr\rbrace.
 	\end{equation*}
 \end{definition}
 The directional derivative of a continuously differentiable function $f:\Omega\rightarrow \mathbb{R}$ in the direction $\mathbf{v}\in \mathbb{R}^2$ with $\abs{\mathbf{v}}_2=1$ in $\mathbf{x}\in \Omega$ is given by
  \begin{equation*}
  	\partial_\mathbf{v}f(\mathbf{x})\mathrel{\mathop:}=\partial_\mathbf{v}[f](\mathbf{x})=\mathbf{v}^{\mathrm{T}}\,\mathrm{grad}\,f(\mathbf{x}).
  \end{equation*}
  For $f\in C^q(\Omega)$ and $0\leq m\leq q$ there exist the directional derivatives of $m$-th order in every direction $\mathbf{v}\in \mathbb{R}^2$ with $\abs{\mathbf{v}}_2=1$. They are given by $\partial^{0}_{\mathbf{v}}f(\mathbf{x})=f(\mathbf{x})$ and
  \begin{equation}\label{eq:m_te_richtungsableitung_allgemein}
  	\partial^{m}_{\mathbf{v}}f(\mathbf{x})\mathrel{\mathop:}=\partial^{m}_{\mathbf{v}}[f](\mathbf{x})=\partial \mathbf{v}\left[ \partial^{m-1}_{\mathbf{v}}f\right](\mathbf{x})=\sum_{\abs{\mathbf{r}}_1=m}\binom{m}{\mathbf{r}}\,\mathbf{v}^{\mathbf{r}}\partial^{\mathbf{r}}f(\mathbf{x}),\qquad 1\leq m\leq q,
  \end{equation}
  where the last equality can be shown by induction. An important tool for the analysis of cartoon-like functions is the decomposition on dyadic squares \cite{candes:curvelets,labate:sparse,schober:detection}. For $j\in \mathbb{N}_0$ let $\mathcal{Q}_j$ be the set of all dyadic squares $Q\subseteq[-\pi,\pi)^2$ with
  \begin{equation*}
  	Q=\left[2\pi k_1\,2^{-j/2}-\pi,2\pi (k_1+1)\,2^{-j/2}-\pi\right)\times\left[2\pi k_2\,2^{-j/2}-\pi,2\pi (k_2+1)\,2^{-j/2}-\pi\right)
  \end{equation*} 
  for $k_1,k_2=0,\dots ,2^{j/2}-1$. For smooth functions $\phi\in C_0^\infty\left(\mathbb{R}^2\right)$ with $\mathrm{supp}\,\phi\subset(-\pi,\pi)^2$ and $Q\in \mathcal{Q}_j$ we define
  \begin{equation}\label{eq:phi_Q}
   \phi_Q(\mathbf{x})\mathrel{\mathop:}=\phi\left(2^{j/2}(x_1+\pi)-\pi(2k_1-1),2^{j/2}(x_2+\pi)-\pi(2k_2-1)\right)
  \end{equation} 
  and assume that $\phi$ defines a smooth partition of unity
  \begin{equation}\label{eq:zerl_der_eins}
  	\sum_{Q\in \mathcal{Q}_j}\phi_Q(\mathbf{x})=1,\qquad \mathbf{x}\in [-\pi,\pi)^2.
  \end{equation}
 Let $T\in \mathrm{STAR}^2(\tau)$ be given. We say $Q\in \mathcal{Q}_j^1\subset \mathcal{Q}_j$ if $\partial T\cap Q\neq\emptyset$ and for the non-intersecting squares we write $\mathcal{Q}_j^0\mathrel{\mathop:}=\mathcal{Q}_j\setminus \mathcal{Q}_j^1$.\\
  
 	For Lebesgue measurable sets $A\subseteq \mathbb{R}^2$ and functions $f:A\rightarrow \mathbb{R}$ we define
  \begin{equation*}
  	\norm{f}_{A,p}\mathrel{\mathop:}=\left( \int_{A}\abs{f(\mathbf{x})}^p\,\mathrm{d}\mathbf{x} \right)^{1/p},\qquad 1\leq p<\infty,
  \end{equation*} 
 and denote the collection of functions satisfying $\norm{f}_{A,p}<\infty$ by $L_p(A)$. For two-dimensional $2\pi$-periodic functions $f:\mathbb{T}^2\rightarrow \mathbb{R}$ given on the torus $\mathbb{T}^2\mathrel{\mathop:}=\mathbb{R}^2\setminus2\pi\,\mathbb{Z}^2$ the usual inner product of the Hilbert space $L_2(\mathbb{T}^2)$ is given by
 \begin{equation*}
 	\langle f,g \rangle_2 \mathrel{\mathop:}=\frac{1}{2\pi}\int_{\mathbb{T}^2}f(\mathbf{x})\overline{g(\mathbf{x})}\,\mathrm{d}\mathbf{x},\qquad\qquad f,g\in L_2(\mathbb{T}^2).
 \end{equation*}
 \begin{figure}[t]\hspace{-0.2cm}
 	\subfloat{
 	{\includegraphics[width=.47\textwidth]{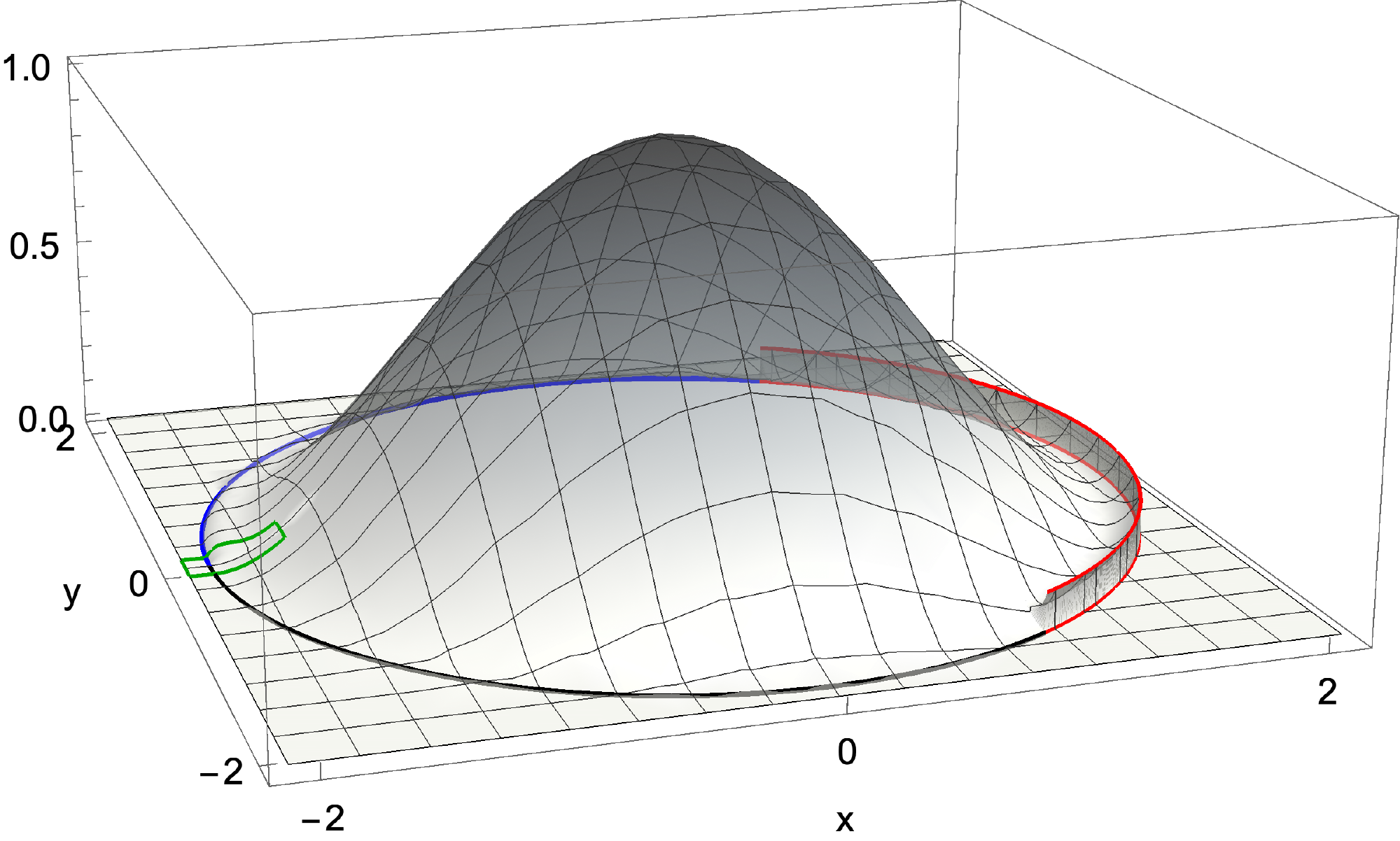}}
 	}\hspace{0.3cm}
 	\subfloat{
 	{\includegraphics[width=.47\textwidth]{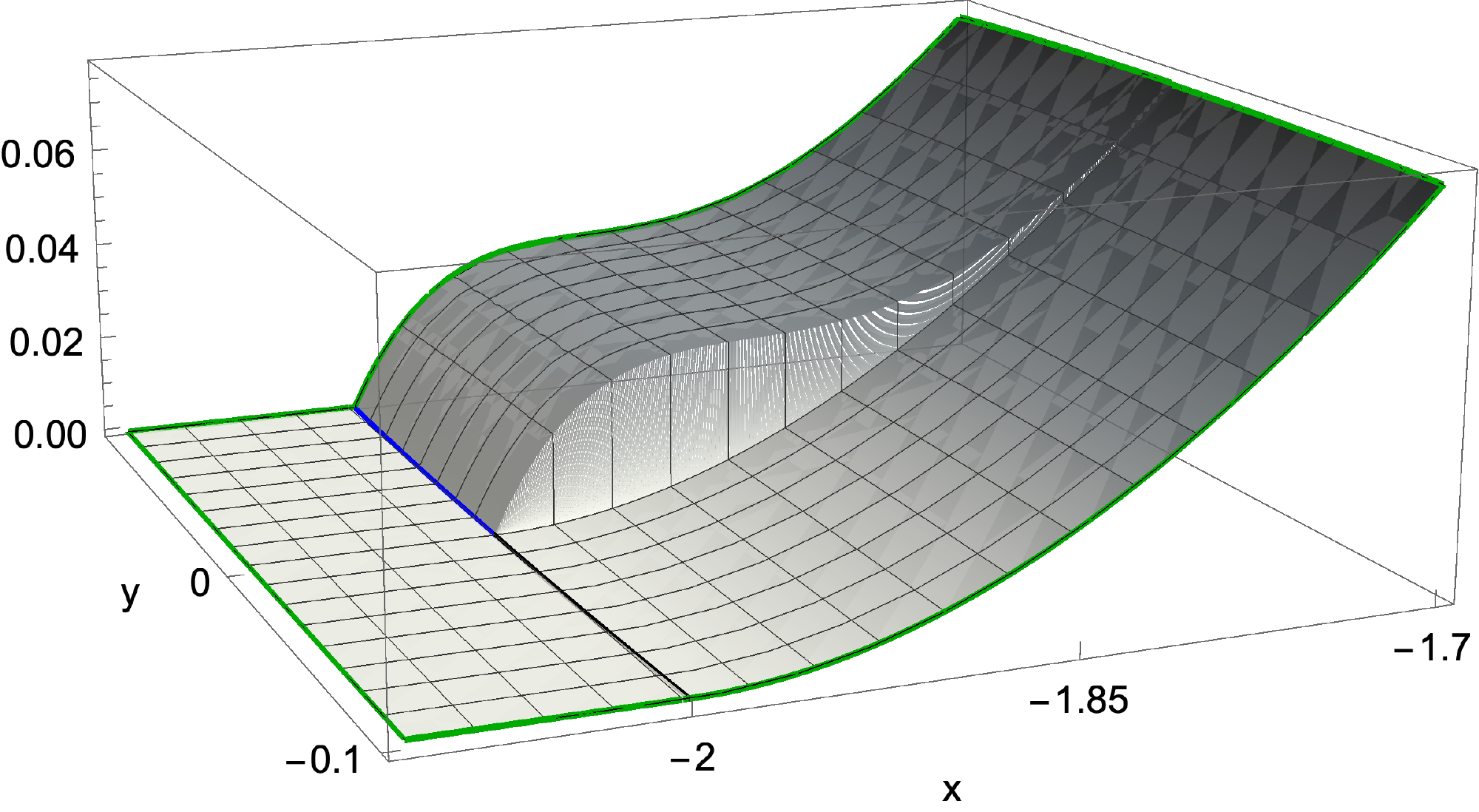}}
 	} 
   	\caption[]{Left: Cartoon-like function with jump discontinuities in the zeroth (red), first (blue) and second (black) order directional derivative on a circle with radius $2$. Right: Zoom into the green window in the left picture.}\label{fig:numerics1}
     \end{figure}
   
 For all $\mathbf{x}\in\partial T$ let $\mathbf{n}(\mathbf{x})=(\cos(\vartheta_\mathbf{x}),\sin(\vartheta_\mathbf{x}))^{\mathrm{T}},\,\vartheta_\mathbf{x}\in[0,2\pi),$ be the outer normal direction of $\partial T$ in $\mathbf{x}$. For the two main theorems we need cartoon-like functions $\mathfrak{f}\in \mathcal{E}^{u+1}(\tau)$ from \cref{def:cartoon_funktionen} with $u>4$ and their $2\pi$-periodization $\mathfrak{f}^{2\pi}$. For a window function $\Psi^{(\mathfrak{i})}\in \mathcal{W}^{2q}$ with $2q\geq u$ and $\mathfrak{i}\in \left\lbrace \mathfrak{h},\mathfrak{v} \right\rbrace$ from \cref{eq:window_function} let $\psi_{j,\ell,\mathbf{y}}^{(\mathfrak{i})}$ be a trigonometric shearlet from \cref{eq:trigonometric_shearlets}.
 
 % Theorem: Main theorem
 \begin{theorem}
 \label{thm:hauptresultat} 
 Let $j\in \mathbb{N}$ be sufficiently large and even and $\ell\in \mathbb{Z}$ with $\abs{\ell}<2^{j/2}$ and $\mathbf{y}\in \mathcal{P}(\mathbf{N}_{j,\ell}^{(\mathfrak{i})})$ be given. For every $Q\in \mathcal{Q}_j^1$ we choose $\mathbf{x}_0\mathrel{\mathop:}=\mathbf{x}_0(Q)\in\partial T \cap Q$. Moreover, let $n\mathrel{\mathop:}=n(Q)\in \mathbb{N}_0$ with $n<u$ such that
 \begin{align*}
		\partial_{\boldsymbol{\Theta}(\vartheta)}^m[f_1](\mathbf{x})=0\quad\text{and}\quad \partial_{\boldsymbol{\Theta}(\vartheta)}^n[f_1](\mathbf{x})\neq 0, & \quad\text{if }\, 0\leq m < n,\\
		f_1(\mathbf{x})\neq 0, & \quad\text{if }\, n=0,
 \end{align*}
is fulfilled for all $\mathbf{x}\in\partial T \cap Q$ and all $\vartheta\in\left(\theta_{j,\ell-2}^{(\mathfrak{i})},\theta_{j,\ell+2}^{(\mathfrak{i})}\right)$. Then there exists a constant $C_1>0$ such that
	\begin{equation*}
	\abs{\left\langle \mathfrak{f}^{2\pi},\psi_{j,\ell,\mathbf{y}}^{(\mathfrak{i})}\right\rangle_2}\leq C_1\,2^{-3j/4}\sum_{Q\in\mathcal{Q}_j^1}\frac{\left(1+2^{j/2}\abs{\sin(\theta_{j,\ell}^{(\mathfrak{i})}-\vartheta_{\mathbf{x}_0})}\right)^{-5/2}}{2^{jn}\Bigl(1+2^j\abs{2\pi\mathbf{y}-\mathbf{x}_0}_2^2\Bigr)^{q}},
	\end{equation*} 
	where $C_1=C_1(\mathfrak{f},\Psi^{(\mathfrak{i})},T)$ is independent of $j,\ell$ and $\mathbf{y}$.
 \end{theorem}
  For $\varepsilon>0$, $T\in \mathrm{STAR}^2(\tau)$ and $\mathbf{y}\in \mathcal{P}(\mathbf{N}_{j,\ell}^{(\mathfrak{i})})$ we define
  \begin{equation}\label{eq:U_epsilon}
  	U_\varepsilon(\mathbf{y})\mathrel{\mathop:}=U_{\varepsilon,T}(\mathbf{y})\mathrel{\mathop:}=\partial T\cap B_\varepsilon(2\pi\mathbf{y}).
  \end{equation}
  % Theorem: main result 2 
  \begin{theorem}
  \label{thm:hauptresultat2}
  Let $0<\varepsilon_0\leq 1$ and a sufficiently large and even $j\in \mathbb{N}$, $\ell\in \mathbb{Z}$ with $\abs{\ell}<2^{j/2}$ and $\mathbf{y}\in \mathcal{P}(\mathbf{N}_{j,\ell}^{(\mathfrak{i})})$ be given. Moreover, we assume the following conditions:
  \begin{itemize}
  	\item [i)] For $\varepsilon=\varepsilon_0\,2^{-j/2}$ there exists $\mathbf{x}_0\in U_\varepsilon(\mathbf{y})$ with $\vartheta_{\mathbf{x}_0}\in\left(\theta_{j,\ell-2}^{(\mathfrak{i})},\theta_{j,\ell+2}^{(\mathfrak{i})}\right)$. 
	\item [ii)] For $n\in \mathbb{N}_0$ with $4(n+1)<u$ we have
  \begin{align} \label{eq:mainthm2}
 		\partial_{\boldsymbol{\Theta}(\vartheta)}^m[f_1](\mathbf{x})=0\quad\text{and}\quad \partial_{\boldsymbol{\Theta}(\vartheta)}^n[f_1](\mathbf{x})\neq 0, & \quad\text{if }\, 0\leq m < n,\\\notag
 		f_1(\mathbf{x})\neq 0, & \quad\text{if }\, n=0,
  \end{align}
  for all $\mathbf{x}\in U_\varepsilon(\mathbf{y})$ and all $\vartheta\in\left(\theta_{j,\ell-2}^{(\mathfrak{i})},\theta_{j,\ell+2}^{(\mathfrak{i})}\right)$. 
  \end{itemize}
  Then there exists a constant $C_2>0$ such that
		\begin{equation*}
			\abs{\left\langle \mathfrak{f}^{2\pi},\psi_{j,\ell,\mathbf{y}}^{(\mathfrak{i})}\right\rangle_2}\geq C_2\,2^{-j(3/4+n)},
		\end{equation*}
		where $C_2=C_2(\mathfrak{f},\Psi^{(\mathfrak{i})},T,\varepsilon_0)$ is independent of $j,\ell$ and $\mathbf{y}$.
  \end{theorem}
  \begin{remark}
	  The two main results from \cite{schober:detection} can be found in the latter theorems as special cases. In \cref{thm:hauptresultat} we have the result from \cite[Theorem 1]{schober:detection} if $n=0$ and in \cref{thm:hauptresultat2} we have the result from \cite[Theorem 2]{schober:detection} if $n=0$, $\mathfrak{f}=\chi_T$ from \cref{def:cartoon_funktionen} is a characteristic function which means $f_0=0$ and $f_1=1$. It should be mentioned that the lower bound from \cref{thm:hauptresultat2} still holds holds true if the condition \cref{eq:mainthm2} is reduced to the nonzero-condition only for the $n$-th order directional derivative.
  \end{remark} 
   
 	\begin{figure}[t]\hspace{-0.2cm}
 		\subfloat{
 		{\includegraphics[width=.34\textwidth]{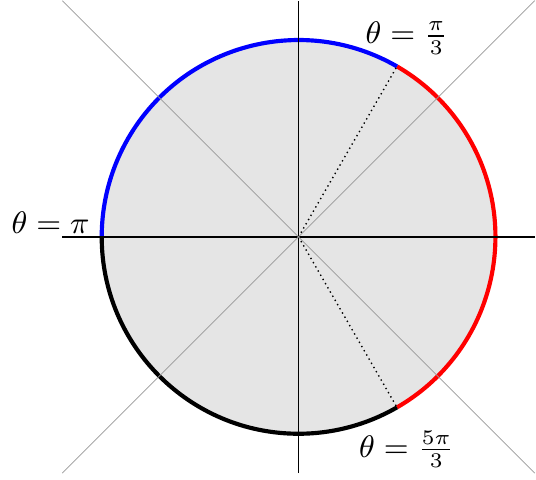}}
 		}\hspace{0.3cm}
 		\subfloat{
 		{\includegraphics[width=.61\textwidth]{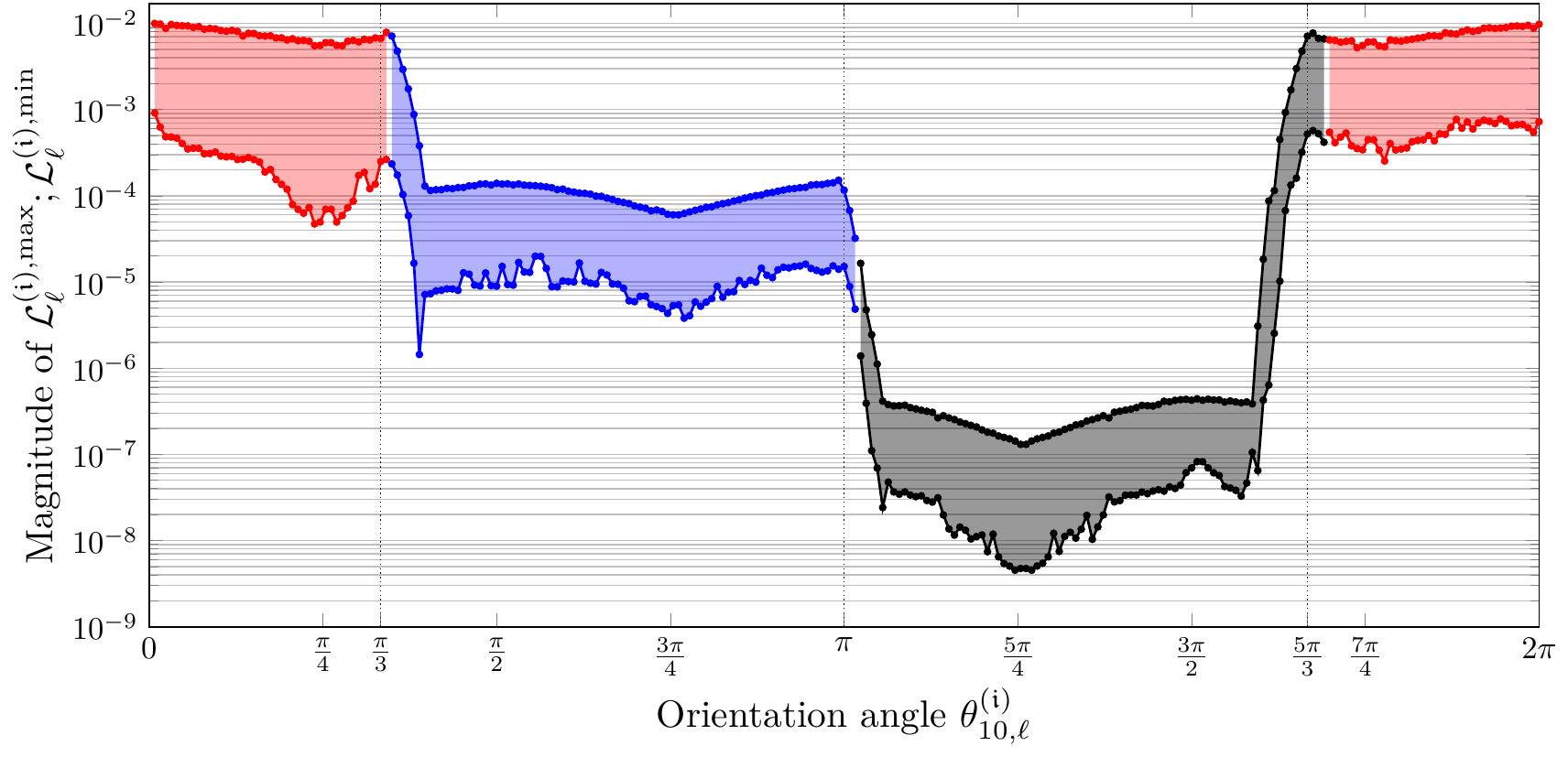}}
 		}
 	  	\caption[]{Left: Schematic visualization of the function from \cref{fig:numerics1} with colored boundary lines where the function has directional jump discontinuities of different orders. Right: Magnitudes of $\mathcal{L}^{(\mathfrak{i}),\mathrm{max}}_{\ell}$ and $\mathcal{L}^{(\mathfrak{i}),\mathrm{min}}_{\ell}$ from \cref{eq:L_l} as functions of the orientation angles $\theta_{10,\ell}^{(\mathfrak{i})}$.}\label{fig:numerics}
 	    \end{figure}
At the end of this section, we include a small numerical example to visualize the main results. We construct a cartoon-like function $\mathfrak{f}$ with directional jump discontinuities of different orders on a circle with radius $2$ (see \cref{fig:numerics1}). Using the parametrization $(2\cos\theta,\,2\sin\theta)^{\mathrm{T}}$, the parts of different smoothness on the boundary are separated at the angles $\theta\in \left\lbrace \frac{\pi}{3},\pi,\frac{5\pi}{3}\right\rbrace$ (see \cref{fig:numerics}). On the red line of the boundary, the function is discontinuous. On the blue line, the function has a jump discontinuity in the first directional derivative and on the black line in the second directional derivative in every direction unless the tangent direction. 
In this example, we choose $j=10$, $\varepsilon_0=\frac{1}{2}$ thus $\varepsilon=\frac{1}{2}\,2^{-5}=\frac{1}{64}$ and consider the matrix $\mathbf{M}_{10}=2^{10}\,\mathbf{I}_2$ with the two-dimensional identity matrix $\mathbf{I}_2$. We collect all pattern points $\mathbf{y}\in \mathcal{P}(\mathbf{M}_{10})$ for which there exists $\mathbf{x}_0\in U_\varepsilon(\mathbf{y})$ fulfilling $\vartheta_{\mathbf{x}_0}\in\left(\theta_{10,\ell-2}^{(\mathfrak{i})},\theta_{10,\ell+2}^{(\mathfrak{i})}\right)$ in the set $\mathcal{Y}_{\ell}^{(\mathfrak{i})}$ where $\ell\in \mathbb{Z}$ with $\abs{\ell}<32$ and $\mathfrak{i}\in \left\lbrace \mathfrak{h},\mathfrak{v} \right\rbrace$. We compute the values 
	\begin{equation}\label{eq:L_l}
		\mathcal{L}^{(\mathfrak{i}),\mathrm{max}}_{\ell}\mathrel{\mathop:}=\max\limits_{\mathbf{y}\in \mathcal{Y}_{\ell}^{(\mathfrak{i})}}\abs{\left\langle \mathfrak{f}^{2\pi},\psi_{j,\ell,\mathbf{y}}^{(\mathfrak{i})}\right\rangle_2},\quad \mathcal{L}^{(\mathfrak{i}),\mathrm{min}}_{\ell}\mathrel{\mathop:}=\min\limits_{\mathbf{y}\in \mathcal{Y}_{\ell}^{(\mathfrak{i})}}\abs{\left\langle \mathfrak{f}^{2\pi},\psi_{j,\ell,\mathbf{y}}^{(\mathfrak{i})}\right\rangle_2} 
	\end{equation} 
	and present them in the right picture of \cref{fig:numerics} as functions of the orientation angles $\theta_{10,\ell}^{(\mathfrak{i})}$. One can clearly see how the magnitude of the shearlet coefficients $\abs{\left\langle \mathfrak{f}^{2\pi},\psi_{j,\ell,\mathbf{y}}^{(\mathfrak{i})}\right\rangle_2}$ depends on the number of vanishing directional derivatives of the function $\mathfrak{f}^{2\pi}$ on the boundary curve as it was anticipated in the main results.
		
   % Section: Main results (end)
 
 % Section Proof of Theorem 3.1
 \section{Proof of Theorem 3.1} 
 \label{sec:proof_of_theorem_3_1}
 The Fourier coefficients of a function $f\in L_1(\mathbb{T}^2)$ are given by
 \begin{equation*}
 	c_{\mathbf{k}}(f)\mathrel{\mathop:}=(2\pi)^{-2}\int_{\mathbb{T}^2}f(\mathbf{x})\,\mathrm{e}^{-\mathrm{i}\mathbf{k}^{\mathrm{T}}\mathbf{x}}\,\mathrm{d}\mathbf{x},\qquad\mathbf{k}\in \mathbb{Z}^2
 \end{equation*}
 and the Fourier transform of $f\in L_1(\mathbb{R}^2)$ is defined as
 	\begin{equation*}
 		\mathcal{F}[f](\mathbf{x})\mathrel{\mathop:}=\mathcal{F}f(\mathbf{x})\mathrel{\mathop:}=(2\pi)^{-2}\int_{\mathbb{R}^2}f(\boldsymbol{\xi})\,\mathrm{e}^{-\mathrm{i}\boldsymbol{\xi}^{\mathrm{T}}\mathbf{x}}\,\mathrm{d}\boldsymbol{\xi},\qquad\mathbf{x}\in \mathbb{R}^2,
 	\end{equation*}
 	and we have the operator
 	\begin{equation*}
 		\mathcal{F}^{-1}[f](\mathbf{x})\mathrel{\mathop:}=\mathcal{F}^{-1}f(\mathbf{x})\mathrel{\mathop:}=\int_{\mathbb{R}^2}f(\boldsymbol{\xi})\,\mathrm{e}^{\mathrm{i}\boldsymbol{\xi}^{\mathrm{T}}\mathbf{x}}\,\mathrm{d}\boldsymbol{\xi},\qquad\mathbf{x}\in \mathbb{R}^2. 
 	\end{equation*}
 Let $q\in \mathbb{N}_0$ and $\mathbf{r}\in \mathbb{N}_0^2$ with $\abs{\mathbf{r}}_1\leq q$. If $f\in L_1(\mathbb{R}^2)$ and $(\mathrm{i}\,\circ)^q\,f\in L_1(\mathbb{R}^2)$, then $\mathcal{F}f\in C^q(\mathbb{R}^2)$ and
 		\begin{equation}\label{eq:properties_fourier2}
 			\partial^\mathbf{r}\mathcal{F}f(\boldsymbol{\xi})=\mathcal{F}\left[(\mathrm{i}\,\circ)^\mathbf{r}\,f(\mathbf{x})\right](\boldsymbol{\xi}).
 		\end{equation}
 		For $f\in C^q(\mathbb{R}^2)$ and $\partial^\mathbf{r}f\in L_1(\mathbb{R}^2)$ we have
 		\begin{equation}\label{eq:properties_fourier3}
 			\mathcal{F}\left[ \partial^\mathbf{r}f \right](\boldsymbol{\xi})=(\mathrm{i}\,\boldsymbol{\xi})^\mathbf{r}\,\mathcal{F}f(\boldsymbol{\xi})
 		\end{equation}
 	   and from \cref{eq:m_te_richtungsableitung_allgemein} and \cref{eq:properties_fourier3}, it follows that the Fourier transform of the $m$-th order directional derivative of a function along a normalized direction $\mathbf{v}\in \mathbb{R}^2$ can be written as
 	    \begin{equation}\label{eq:eigenschaft_fourier_richtungsableitung} 	\mathcal{F}\left[\partial^m_{\mathbf{v}}f\right](\boldsymbol{\xi})=\sum_{\abs{\mathbf{r}}_1=m}\binom{m}{\mathbf{r}}\mathbf{v}^\mathbf{r}\mathcal{F}\left[\partial^{\mathbf{r}}f\right](\boldsymbol{\xi})=\mathrm{i}^m\mathcal{F}f(\boldsymbol{\xi})\sum_{\abs{\mathbf{r}}_1=m}\binom{m}{\mathbf{r}}\mathbf{v}^{\mathbf{r}}\boldsymbol{\xi}^\mathbf{r}=\mathrm{i}^m (\mathbf{v}^{\mathrm{T}}\boldsymbol{\xi})^m\mathcal{F}f(\boldsymbol{\xi}).
 	    \end{equation}
For the remainder of this section we fix the function $\phi\in C_0^\infty\left(\mathbb{R}^2\right)$ with $\mathrm{supp}\,\phi\subset(-\pi,\pi)^2$ and consider its scaled version $\phi_j\mathrel{\mathop:}=\phi\left(2^{j/2}\circ\right)$. Following the approach from \cite[Chapter 6.1]{candes:curvelets}, we assume that for sufficiently large $j\geq j_0$ the edge curve $\partial T$ can be parametrized on the support of $\phi_Q,\,Q\in \mathcal{Q}_j^1,\,$ either as $(x_1,E(x_1))^{\mathrm{T}}$ or $(E(x_2),x_2)^{\mathrm{T}}$.
\begin{definition}
	For $x_2\in\left[-2^{-j/2},2^{-j/2}\right]$ let $(E(x_2),x_2)^{\mathrm{T}}$ be a parametrization of $\partial T$ with $E(0)=E'(0)=0$. For $f\in C^2(\mathbb{R}^2)$ we call
	\begin{equation*}
		\mathcal{K}_j(\mathbf{x})=f(\mathbf{x})\,\phi_j(\mathbf{x})\,\chi_{\lbrace \mathbf{x}\,:\, x_1\geq t(x_2)\rbrace}(\mathbf{x})
	\end{equation*}
	standard edge fragment.
\end{definition}
Let $\mathcal{K}_{j,\mathbf{x}_0,\vartheta}$ be an arbitrary edge fragment, which means that the tangent in $\mathbf{x}_0\in\partial T$ is pointing in the direction $\boldsymbol{\Theta}(\vartheta)=(\cos{\vartheta},\sin{\vartheta})^{\mathrm{T}}$ for $\vartheta\in[0,2\pi)$. Then $\mathcal{K}_{j,\mathbf{0},0}=\mathcal{K}_j$ is a standard edge fragment and it was shown in \cite[Corollary 6.7]{candes:curvelets} that the corresponding Fourier transform fulfills
\begin{equation}\label{eq:rotate_edge_fragment}
	\mathcal{F}\mathcal{K}_{j,\mathbf{x}_0,\vartheta}(\boldsymbol{\xi})=\mathrm{e}^{-\mathrm{i}\,\mathbf{x}_0^{\mathrm{T}}\boldsymbol{\xi}}\,\mathcal{F}\mathcal{K}_j(\mathbf{R}_\vartheta^{\mathrm{T}}\,\boldsymbol{\xi}),
\end{equation}
where $\mathbf{R}_\vartheta$ is a rotation matrix by the angle $\vartheta$. Here we show the following lemma which generalizes \cite[Lemma 6]{schober:detection}.
		      \begin{lemma} 
		      \label{lem:norm_FT_Q1}
		      	For $\mathfrak{i}\in\lbrace \mathfrak{h},\mathfrak{v}\rbrace$ and $q\in \mathbb{N}_0$ let $\Psi^{(\mathfrak{i})}\in \mathcal{W}^q$ be given. Moreover, let $\mathcal{K}_{j,\mathbf{0},\vartheta}$ with $\vartheta\in \left( \theta_{j,\ell-2}^{(\mathfrak{i})},\theta_{j,\ell+2}^{(\mathfrak{i})} \right)$ be an arbitrary edge fragment and $f_j=f\,\phi_j$ a function with $\partial_{\boldsymbol{\Theta}(\theta_{j,\ell}^{(\mathfrak{i})})}^n f_j=\mathcal{K}_{j,\mathbf{0},\vartheta}$ for $n\in \mathbb{N}_0$. Then for $\mathbf{r}\in \mathbb{N}_0^2$ we have
		      	\begin{equation*}	
		\norm{\partial^{\mathbf{r}}\left[\mathcal{F}\left[f_j\right]\,\Psi_{j,\ell}^{(\mathfrak{i})}\right]}^2_{\mathrm{supp}\,\Psi_{j,\ell}^{(\mathfrak{i})},2}\leq C(n,\mathbf{r})\,2^{-j(3/2+2n+\abs{\mathbf{r}}_1)}\,\left(1+2^{j/2}\abs{\sin(\theta_{j,\ell}^{(\mathfrak{i})}-\vartheta)}\right)^{-5}.
		      	\end{equation*}
		      \end{lemma}
  		\begin{proof}
 			We use an idea from \cite[Corollary 6.6]{candes:curvelets} and define $\phi_\mathbf{r}(\mathbf{x})\mathrel{\mathop:}=\mathbf{x}^\mathbf{r}\,\phi(\mathbf{x})$. It follows that $\phi_\mathbf{r}\left(2^{j/2}\circ\right)\in C^\infty_0\left(\mathbb{R}^2\right)$ and $\abs{\mathrm{supp}\,\phi_\mathbf{r}\left(2^{j/2}\circ\right)}\leq 2^{-j}$. We obtain the representation
 			\begin{equation*}
				\mathbf{x}^\mathbf{r}f_j(\mathbf{x})=2^{-j\abs{\mathbf{r}}_1/2}\,f(\mathbf{x})\,\phi_\mathbf{r}\left(2^{j/2}\,\mathbf{x}\right)=2^{-j\abs{\mathbf{r}_1}/2}f_{j,\mathbf{r}}(\mathbf{x}),
 			\end{equation*}
 			where $f_{j,\mathbf{r}}\mathrel{\mathop:}=f\,\phi_\mathbf{r}\left(2^{j/2}\,\circ\right)$. Note that the function $f_{j,\mathbf{r}}$ also fulfills $\partial_{\boldsymbol{\theta}(\theta_{j,\ell}^{(\mathfrak{i})})}^n f_{j,\mathbf{r}}=\mathcal{K}_{j,\mathbf{0},\vartheta}$. Using \cref{eq:properties_fourier2} we get
	 		\begin{equation}\label{proof:estimate_f_smooth5}
	 			\partial^\mathbf{r}\mathcal{F}f_j(\boldsymbol{\xi})=\mathcal{F}\left[( \mathrm{i}\,\circ)^\mathbf{r}f_j(\circ)\right](\boldsymbol{\xi})=\mathrm{i}^\mathbf{r}\,2^{-j\abs{\mathbf{r}}_1/2}\,\mathcal{F}\left[f(\circ)\,\phi_\mathbf{r}\left(2^{j/2}\circ\right)\right](\boldsymbol{\xi})
	 		\end{equation}
	 		and with \cref{proof:estimate_f_smooth5} and \cref{eq:eigenschaft_fourier_richtungsableitung} we have
 			\begin{align}
				\int\limits_{\mathrm{supp}\Psi_{j,\ell}^{(\mathfrak{i})}}\bigl\lvert\partial^{\mathbf{r}}\left[\mathcal{F}f_j\right](\boldsymbol{\xi})\bigr\rvert^2\mathrm{d}\boldsymbol{\xi}&=2^{-j\abs{\mathbf{r}}_1}\int\limits_{\mathrm{supp}\Psi_{j,\ell}^{(\mathfrak{i})}}\bigl\lvert\mathcal{F}f_{j,\mathbf{r}}(\boldsymbol{\xi})\bigr\rvert^2 \mathrm{d}\boldsymbol{\xi}\notag\\\label{proof:int_FT_Q1_0}
 				&=2^{-j\abs{\mathbf{r}}_1}\int\limits_{\mathrm{supp}\Psi_{j,\ell}^{(\mathfrak{i})}}\left\lvert\left( \boldsymbol{\Theta}^{\mathrm{T}}(\theta_{j,\ell}^{(\mathfrak{i})})\,\boldsymbol{\xi} \right)^{-n}\mathcal{F}[\mathcal{K}_{j,\mathbf{0},\vartheta}](\boldsymbol{\xi})\right\rvert^2 \mathrm{d}\boldsymbol{\xi}.
 			\end{align}
 			In the following, we need a result from \cite[Theorem 6.1]{candes:curvelets} given by
 	 		\begin{equation}\label{eq:int_FT_radius}
 	 			\int\limits_{\abs{\rho}\in I_j}\abs{\mathcal{F}\mathcal{K}_j\Bigr(\rho\,\boldsymbol{\Theta}(\theta-\vartheta)\Bigl)}^2 \mathrm{d}\rho\leq C\,2^{-2j}\,\Bigl(1+2^{j/2}\abs{\sin\left(\theta-\vartheta\right)}\Bigr)^{-5},
 	 	 	\end{equation}
 			where $I_j=\left[2^{j-1},2^{j+1}\right]$. In polar coordinates $\boldsymbol{\xi}=\rho\,\boldsymbol{\Theta}(\theta)$ with $\rho=\abs{\boldsymbol{\xi}}_2$ the inner product from \cref{proof:int_FT_Q1_0} fulfills
 			\begin{equation}\label{proof:int_FT_Q1_1}
 				\abs{\boldsymbol{\Theta}^{\mathrm{T}}(\theta_{j,\ell}^{(\mathfrak{i})})\,\boldsymbol{\xi}}=\abs{\rho\,\cos\left( \theta_{j,\ell}^{(\mathfrak{i})}-\theta\right)}\geq C_2\abs{\rho},
 			\end{equation}
 			if $\theta\in \left( \theta_{j,\ell-2}^{(\mathfrak{i})},\theta_{j,\ell+2}^{(\mathfrak{i})} \right)$. Additionally, we have
 			\begin{equation}\label{proof:int_FT_Q1_2}
 				\Bigl(\theta_{j,\ell+2}^{(\mathfrak{i})}-\theta_{j,\ell-2}^{(\mathfrak{i})}\Bigr)\leq C_3\,2^{-j/2}.
 			\end{equation}
 			We transform the integral from \cref{proof:int_FT_Q1_0} into polar coordinates and use \cref{eq:int_FT_radius}, \cref{proof:int_FT_Q1_1}, \cref{proof:int_FT_Q1_2} and \cite[Lemma 1]{schober:detection} to finally get 
  			\begin{align}
 \int\limits_{\mathrm{supp}\Psi_{j,\ell}^{(\mathfrak{i})}}\bigl\lvert\partial^{\mathbf{r}}\left[\mathcal{F}f_j\right](\boldsymbol{\xi})\bigr\rvert^2 \mathrm{d}\boldsymbol{\xi}&=2^{-j\abs{\mathbf{r}}_1}\int\limits_{\theta_{j,\ell-2}^{(\mathfrak{i})}}^{\theta_{j,\ell+2}^{(\mathfrak{i})}}\int\limits_{\frac{2^j}{3}}^{2^{j+1}}\abs{\left( \rho\,\cos\left( \theta_{j,\ell}^{(\mathfrak{i})}-\theta\right) \right)^{-n}\mathcal{F}\mathcal{K}_j\Bigr(\rho\,\boldsymbol{\Theta}(\theta-\vartheta)\Bigl)}^2\rho\, \mathrm{d}\rho\,\mathrm{d}\theta\notag\\
  				&\leq C_4(n,\mathbf{r})\,2^{-j(1+2n+\abs{\mathbf{r}}_1)}\int\limits_{\theta_{j,\ell-2}^{(\mathfrak{i})}}^{\theta_{j,\ell+2}^{(\mathfrak{i})}}\left(1+2^{j/2}\bigl\lvert\sin(\theta-\vartheta)\bigr\rvert\right)^{-5}\mathrm{d}\theta\notag\\\label{proof:int_FT_Q1_3}
  				&\leq C_5(n,\mathbf{r})\,2^{-j(3/2+2n+\abs{\mathbf{r}}_1)}\left(1+2^{j/2}\abs{\sin(\theta_{j,\ell}^{(\mathfrak{i})}-\vartheta)}\right)^{-5}.
  			\end{align}
			To obtain the desired estimate, we repeat the steps from the proof of \cite[Lemma 4]{schober:detection} and apply \cref{proof:int_FT_Q1_3}.
  		\end{proof}
		In the following, we consider the second order differential operator $L\mathrel{\mathop:}=I+2^j\Delta$ used in \cite{candes:curvelets,labate:sparse,schober:detection} where $\Delta\mathrel{\mathop:}=\partial^{(2,0)}+\partial^{(0,2)}$ is the Laplace operator. The following Lemma is a generalization of \cite[Lemma 8]{schober:detection}. The proof is similar and will be omitted.
      \begin{lemma}\label{lem:norm_Lq}
  		For $\mathfrak{i}\in\lbrace \mathfrak{h},\mathfrak{v}\rbrace$ and $q\in \mathbb{N}_0$ let $\Psi^{(\mathfrak{i})}\in \mathcal{W}^{2q}$ be given. We consider functions $f_{j,0}\mathrel{\mathop:}=f_0\,\phi_j$ with $f_0\in C_0^u\left(\mathbb{R}^2\right)$ for $u\in \mathbb{N}_0$ and for an arbitrary edge fragment $\mathcal{K}_{j,\mathbf{0},\vartheta}$ with $\vartheta\in \left( \theta_{j,\ell-2}^{(\mathfrak{i})},\theta_{j,\ell+2}^{(\mathfrak{i})} \right)$ let $f_{j,1}\mathrel{\mathop:}=f_1\,\phi_j$ such that $\partial_{\boldsymbol{\Theta}(\theta_{j,\ell}^{(\mathfrak{i})})}^n f_{j,1}=\mathcal{K}_{j,\mathbf{0},\vartheta}$ for $n\in \mathbb{N}_0$. Then there exist constants $C_1(u,q),C_2(n,q)>0$ such that
  		\begin{equation*}
  			\norm{L^q\left[ \mathcal{F}[h]\,\Psi_{j,\ell}^{(\mathfrak{i})} \right]}^2_{\mathrm{supp}\,\Psi_{j,\ell}^{(\mathfrak{i})},2}\leq 
  			\begin{cases}
  				C_1(u,q)\,2^{-j(2u+1)}, & \text{if  } h=f_{j,0},\vspace{0.3cm}\\
  				\dfrac{C_2(n,q)\,2^{-j(3/2+2n)}}{\left(1+2^{j/2}\abs{\sin\left(\theta_{j,\ell}^{(\mathfrak{i})}-\vartheta\right)}\right)^{5}}, & \text{if  } h=f_{j,1}.
  			\end{cases}
  		\end{equation*}
      \end{lemma}
	  We are ready to prove the first main theorem of this paper.
	  \begin{proof}[Proof of \cref{thm:hauptresultat}]
	  Let $T\in \mathrm{STAR}^2(\tau)$ and $\mathfrak{f}\mathrel{\mathop:}=f\,\chi_{T}\in\mathcal{E}^{u+1}(\tau)$ be given. Using the smooth functions $\phi_Q\in C_0^{\infty}\left( \mathbb{R}^2 \right),\,Q\in \mathcal{Q}_j,$ from \cref{eq:phi_Q} which form a partition of unity in \cref{eq:zerl_der_eins}, we can decompose the function $\mathfrak{f}$ on dyadic squares as
  \begin{equation}\label{eq:zerl_f_Q}
  	\mathfrak{f}=\sum_{Q\in \mathcal{Q}_j}\mathfrak{f}_Q=\sum_{Q\in \mathcal{Q}_j^0}\mathfrak{f}_Q+\sum_{Q\in \mathcal{Q}_j^1}\mathfrak{f}_Q,
  \end{equation}
  where $\mathfrak{f}_Q\mathrel{\mathop:}=\mathfrak{f}\,\phi_Q$. It was observed in \cite[Section 5.1]{candes:curvelets} that there are constants $C_1,C_2>0$ with
  \begin{equation}\label{eq:maechtigkeit_Q_j}
  	\abs{\mathcal{Q}_j^0}\leq C_1\,2^j,\qquad\qquad\abs{\mathcal{Q}_j^1}\leq C_2\,2^{j/2}.
  \end{equation}
 	 We denote by $\mathfrak{f}_Q^{2\pi}$ the $2\pi$-periodization of $\mathfrak{f}_Q$. Since $\mathfrak{f}_Q\in L_1(\mathbb{R}^2)$ the Fourier coefficients of $\mathfrak{f}_Q^{2\pi}$ can be written as
 \begin{equation*}
 	c_{\mathbf{k}}(\mathfrak{f}_Q^{2\pi})=\mathcal{F}[\mathfrak{f}_Q](\mathbf{k}),\qquad \mathbf{k}\in \mathbb{Z}^2.
 \end{equation*}
  From \cref{eq:properties_fourier2} we get $\mathcal{F}[\mathfrak{f}_Q]\in C^{2q}(\mathbb{R}^2)$ because $\mathfrak{f}_Q$ is compactly supported. Moreover, the assumption $\Psi_{j,\ell}^{(\mathfrak{i})}\in \mathcal{W}^{2q}$ with $2q> 4$ implies $\mathcal{F}[\mathfrak{f}_Q]\,\Psi_{j,\ell}^{(\mathfrak{i})}\in C_0^{2q}(\mathbb{R}^2)$. Thus, we can use the Poisson summation formula and Parseval's identity (see \cite{schober:detection}) to obtain
  \begin{equation}\label{eq:beweis_der_oberen_schranke1}
  	\left\langle \mathfrak{f}_Q^{2\pi},\psi^{(\mathfrak{i})}_{j,\ell,\mathbf{y}}\right\rangle_2=2^{-3j/4}\sum_{\mathbf{k}\in \mathbb{Z}^2}\mathcal{F}[\mathfrak{f}_Q](\mathbf{k})\,\Psi^{(\mathfrak{i})}_{j,\ell}(\mathbf{k})\,\mathrm{e}^{2\pi\mathrm{i}\mathbf{k}^{\mathrm{T}}\mathbf{y}}=2^{-3j/4}\sum_{\mathbf{n}\in \mathbb{Z}^2}S_Q(\mathbf{n}),
  \end{equation}
where 
  \begin{equation*}
 	S_Q(\mathbf{n})\mathrel{\mathop:}=\mathcal{F}^{-1}\left[\mathcal{F}[\mathfrak{f}_Q]\Psi^{(\mathfrak{i})}_{j,\ell} \right]\Bigl(2\pi(\mathbf{y}+\mathbf{n})\Bigr)=\int\limits_{\mathbb{R}^2}\mathcal{F}[\mathfrak{f}_Q](\boldsymbol{\xi})\,\Psi^{(\mathfrak{i})}_{j,\ell}(\boldsymbol{\xi})\,\mathrm{e}^{2\pi\mathrm{i}\boldsymbol{\xi}^{\mathrm{T}}(\mathbf{y}+\mathbf{n})}\,\mathrm{d}\boldsymbol{\xi}.
  \end{equation*}
  We follow some of the steps in the proof of \cite[Theorem 3.1]{schober:detection} and use repeated partial integration and Hölder's inequality to obtain
  \begin{equation}\label{eq:beweis_der_oberen_schranke3}
  	\bigl\lvert S_Q(\mathbf{n})\bigr\rvert\leq 2^{3j/4}\Bigl(1+2^j\abs{2\pi(\mathbf{y}+\mathbf{n})}_2^2\Bigr)^{-q}\norm{L^q\left[\mathcal{F}[\mathfrak{f}_Q]\,\Psi^{(\mathfrak{i})}_{j,\ell}\right] }_{\mathrm{supp}\,\Psi_{j,\ell}^{(\mathfrak{i})},2}.
  \end{equation}
  In \cite[Theorem 3.1]{schober:detection} it was also shown that
   	\begin{equation*}
   		\sum_{\mathbf{n}\in \mathbb{Z}^2\setminus\{\mathbf{0}\}}\Bigl(1+2^j\abs{2\pi(\mathbf{y}+\mathbf{n}) }_2^2\Bigr)^{-q}\leq C_2(q)\,2^{-jq},
   	\end{equation*}
 	which leads to
 \begin{equation}\label{eq:beweis_der_oberen_schranke4}
 	2^{-3j/4}\sum_{\mathbf{n}\in \mathbb{Z}^2\setminus\{\mathbf{0}\}}\bigl\lvert S_Q(\mathbf{n})\bigr\rvert\leq C(q)\,2^{-jq}\norm{L^q\left[\mathcal{F}[\mathfrak{f}_Q]\,\Psi^{(\mathfrak{i})}_{j,\ell}\right] }_{\mathrm{supp}\,\Psi_{j,\ell}^{(\mathfrak{i})},2}.
\end{equation}
In the following, we distinguish wether the boundary $\partial T$ intersects the support of $\phi_Q$ or not and consider two different cases.
 \begin{itemize}
 	\item [i)] Let $Q\in \mathcal{Q}_j^0$: \\
	If $Q\cap T=\emptyset$, we have $\mathfrak{f}_Q=0$ and thus
 \begin{equation*}
 	\abs{\left\langle \mathfrak{f}_Q^{2\pi},\psi^{(\mathfrak{i})}_{j,\ell,\mathbf{y}}\right\rangle_2}=0.
 \end{equation*}  
 If $Q\subset T$, we choose $\mathbf{x}_1\in[-\pi,\pi]^2$ with
  \begin{equation}\label{eq:x1}
  	\abs{2\pi \mathbf{y}-\mathbf{x}_1}_2\geq C>0
  \end{equation} 
  and consider the function $\widetilde{\mathfrak{f}}_Q(\mathbf{x})\mathrel{\mathop:}=\mathfrak{f}_Q(\mathbf{x}+\mathbf{x}_1)$. From \cref{eq:rotate_edge_fragment} it follows that $\mathcal{F}[\mathfrak{f}_Q](\boldsymbol{\xi})=\mathrm{e}^{-\mathrm{i}\,\boldsymbol{\xi}^{\mathrm{T}}\mathbf{x}_1}\,\mathcal{F}[\widetilde{\mathfrak{f}}_Q](\boldsymbol{\xi})$ which implies
  \begin{equation*}
 	S_Q(\mathbf{0})=\mathcal{F}^{-1}\left[\mathcal{F}[\mathfrak{f}_Q]\Psi^{(\mathfrak{i})}_{j,\ell} \right](2\pi\mathbf{y})=\int\limits_{\mathbb{R}^2}\mathcal{F}[\widetilde{\mathfrak{f}}_Q](\boldsymbol{\xi})\,\Psi^{(\mathfrak{i})}_{j,\ell}(\boldsymbol{\xi})\,\mathrm{e}^{\mathrm{i}\,\boldsymbol{\xi}^{\mathrm{T}}(2\pi\mathbf{y}-\mathbf{x}_1)}\,\mathrm{d}\boldsymbol{\xi}.
  \end{equation*} 
  Since $\mathcal{F}[\widetilde{\mathfrak{f}}_Q]\,\Psi_{j,\ell}^{(\mathfrak{i})}\in C_0^q(\mathbb{R}^2)$, we can repeat the steps which led to \cref{eq:beweis_der_oberen_schranke3} and use \cref{eq:x1} to obtain
 \begin{equation}\label{eq:beweis_der_oberen_schranke5}
 	\bigl\lvert S_Q(\mathbf{0}) \bigr\rvert\leq2^{-j(q-3/4)}\norm{L^q\left[\mathcal{F}[\widetilde{\mathfrak{f}}_Q]\,\Psi^{(\mathfrak{i})}_{j,\ell}\right] }_{\mathrm{supp}\,\Psi_{j,\ell}^{(\mathfrak{i})},2}.
 \end{equation}
 Finally, the estimates \cref{eq:beweis_der_oberen_schranke4}, \cref{eq:beweis_der_oberen_schranke5} and the first case of \cref{lem:norm_Lq} plugged in \cref{eq:beweis_der_oberen_schranke1} lead to
 \begin{equation}\label{eq:beweis_der_oberen_schranke6}
 	\abs{\left\langle \mathfrak{f}_Q^{2\pi},\psi^{(\mathfrak{i})}_{j,\ell,\mathbf{y}}\right\rangle_2}\leq 2^{-3j/4}\Bigl( \abs{S_Q(\mathbf{0})}+\hspace{-0.3cm}\sum_{\mathbf{n}\in \mathbb{Z}^2\setminus\{\mathbf{0}\}}\bigl\lvert S_Q(\mathbf{n})\bigr\rvert \Bigr)\leq C_1(u,q)\,2^{-j(q+u+3/2)}.
 \end{equation}
\item[ii)] Let $Q\in \mathcal{Q}_j^1$:\\ 
Then we have
  					\begin{equation*}
						S_Q(\mathbf{0})=\int\limits_{\mathbb{R}^2}\mathcal{F}[\mathfrak{f}_Q]\,\Psi^{(\mathfrak{i})}_{j,\ell}(\boldsymbol{\xi})\,\mathrm{e}^{\mathrm{i}\,\boldsymbol{\xi}^{\mathrm{T}}(2\pi\mathbf{y}-\mathbf{x}_0)}\,\mathrm{d}\boldsymbol{\xi},
  					\end{equation*}
  					where $\partial_{\boldsymbol{\Theta}(\theta_{j,\ell}^{(\mathfrak{i})})}^n \mathfrak{f}_Q=\mathcal{K}_{j,\mathbf{0},\vartheta_{\mathbf{x}_0}}$ and $\mathcal{K}_{j,\mathbf{0},\vartheta_{\mathbf{x}_0}}$ is an arbitrary edge fragment. With the same arguments as before we see that
 					\begin{equation}\label{eq:beweis_der_oberen_schranke7}
 						\bigl\lvert S_Q(\mathbf{0}) \bigr\rvert\leq2^{3j/4}\Bigl(1+2^j\abs{2\pi\mathbf{y}-\mathbf{x}_0}_2^2\Bigr)^{-q}\norm{L^q\left[\mathcal{F}[f_Q]\,\Psi^{(\mathfrak{i})}_{j,\ell}\right] }_{\mathrm{supp}\,\Psi_{j,\ell}^{(\mathfrak{i})},2}.
 					\end{equation}
 					From \cref{eq:beweis_der_oberen_schranke4}, \cref{eq:beweis_der_oberen_schranke7} and the second case of \cref{lem:norm_Lq} we deduce 
 \begin{align}\label{eq:beweis_der_oberen_schranke8}
 \abs{\left\langle \mathfrak{f}_Q^{2\pi},\psi^{(\mathfrak{i})}_{j,\ell,\mathbf{y}}\right\rangle_2}&\leq 2^{-3j/4}\Bigl( \abs{S_Q(\mathbf{0})}+\hspace{-0.3cm}\sum_{\mathbf{n}\in \mathbb{Z}^2\setminus\{\mathbf{0}\}}\bigl\lvert S_Q(\mathbf{n})\bigr\rvert \Bigr)\notag\\
&\leq C_3(n,q)\,2^{-3j/4}\frac{\left(1+2^{j/2}\abs{\sin(\theta_{j,\ell}^{(\mathfrak{i})}-\vartheta_{\mathbf{x}_0})}\right)^{-5/2}}{2^{jn}\Bigl(1+2^j\abs{2\pi\mathbf{y}-\mathbf{x}_0}_2^2\Bigr)^{q}}.
 \end{align}
 		With the decomposition in \cref{eq:zerl_f_Q} we can use the estimates in \cref{eq:beweis_der_oberen_schranke6} and \cref{eq:beweis_der_oberen_schranke8} to get
 		\begin{align}\label{eq:beweis_der_oberen_schranke9}
 			\abs{\left\langle \mathfrak{f}^{2\pi},\psi^{(\mathfrak{i})}_{j,\ell,\mathbf{y}}\right\rangle_2}&\leq\sum_{Q\in \mathcal{Q}_j^0}\abs{\left\langle \mathfrak{f}_Q^{2\pi},\psi^{(\mathfrak{i})}_{j,\ell,\mathbf{y}}\right\rangle_2}+\sum_{Q\in \mathcal{Q}_j^1} \abs{\left\langle \mathfrak{f}_Q^{2\pi},\psi^{(\mathfrak{i})}_{j,\ell,\mathbf{y}}\right\rangle_2}\notag\\
 			&\leq C_4(u,n,q)\,2^{-3j/4}\sum_{Q\in\mathcal{Q}_j^1}\frac{\left(1+2^{j/2}\abs{\sin(\theta_{j,\ell}^{(\mathfrak{i})}-\vartheta_{\mathbf{x}_0})}\right)^{-5/2}}{2^{jn}\Bigl(1+2^j\abs{2\pi\mathbf{y}-\mathbf{x}_0}_2^2\Bigr)^{q}}.
 		\end{align}
    	 To finish the proof we consider general cartoon-like functions $\mathfrak{f}_0\in \mathcal{E}^{u+1}(\tau)$ of the form $\mathfrak{f}_0=f_0+f\,\chi_T=f_0+\mathfrak{f}$, where $\mathfrak{f}=f\,\chi_T$, $f_0,f\in C_0^{u+1}\left(\mathbb{R}^2\right)$ and $T\in \mathrm{STAR}^2(\tau)$. For the function $f_0$ we define $f_{0,Q}\mathrel{\mathop:}=f_0\,\phi_Q$ and have the representation
    	 \begin{equation*}
    	 	f_0=\sum_{Q\in \mathcal{Q}_j^0}f_{0,Q},
    	 \end{equation*}
    	 since $\mathcal{Q}_j^1=\emptyset$. With \cref{eq:maechtigkeit_Q_j} and \cref{eq:beweis_der_oberen_schranke6} we get
    	 \begin{equation}\label{eq:f_0}
    	 	\abs{\left\langle f_0^{2\pi},\psi^{(\mathfrak{i})}_{j,\ell,\mathbf{y}}\right\rangle_2}\leq\sum_{Q\in \mathcal{Q}_j^0}\abs{\left\langle f_{0,Q}^{2\pi},\psi^{(\mathfrak{i})}_{j,\ell,\mathbf{y}}\right\rangle_2}\leq C_5(u,q)\,2^{-j(q+u+1/2)}.
    	 \end{equation}
    	The estimates \cref{eq:beweis_der_oberen_schranke9} and \cref{eq:f_0} lead to
 	\begin{align*}
 		\abs{\left\langle \mathfrak{f}_0^{2\pi},\psi^{(\mathfrak{i})}_{j,\ell,\mathbf{y}}\right\rangle_2}&\leq\abs{\left\langle f_0^{2\pi},\psi^{(\mathfrak{i})}_{j,\ell,\mathbf{y}}\right\rangle_2}+\abs{\left\langle \mathfrak{f}^{2\pi},\psi^{(\mathfrak{i})}_{j,\ell,\mathbf{y}}\right\rangle_2}\\
 		&\leq C_6(u,n,q)\,2^{-3j/4}\sum_{Q\in\mathcal{Q}_j^1}\frac{\left(1+2^{j/2}\abs{\sin(\theta_{j,\ell}^{(\mathfrak{i})}-\vartheta_{\mathbf{x}_0})}\right)^{-5/2}}{2^{jn}\Bigl(1+2^j\abs{2\pi\mathbf{y}-\mathbf{x}_0}_2^2\Bigr)^{q}},
 	\end{align*}
 	since $n<u$ and the proof is finished.
\end{itemize}\vspace{-0.8cm}
\end{proof}
% Section proof_of_theorem_3_1 (end)

% Section Localization lemmata
\section{Localization lemmata} 
\label{sec:localization_lemmata}
  
   Let $T\in \mathrm{STAR^2}(\tau)$ and $p\mathrel{\mathop:}=p_u:\mathbb{R}^2\rightarrow \mathbb{R}$ be a bivariate polynomial of order $u$. In \cite[Lemma 4.1]{labate:smooth} the authors showed that the Fourier transform of the function $P_u\mathrel{\mathop:}=p\,\chi_T$ can be written as
   	 \begin{equation}\label{fourier_transformation_gauss}
   		\mathcal{F}P_u(\boldsymbol{\xi})=(2\pi)^{-2}\int\limits_T p(\mathbf{x})\,\mathrm{e}^{-\mathrm{i}\boldsymbol{\xi}^{\mathrm{T}}\mathbf{x}}\,\mathrm{d}\mathbf{x}=\sum\limits_{m=0}^u\frac{C_m}{\abs{\boldsymbol{\xi}}_2^{m+2}}\int\limits_{\partial T}p_m(\mathbf{x},\boldsymbol{\xi})\,\mathrm{e}^{-\mathrm{i}\boldsymbol{\xi}^{\mathrm{T}}\mathbf{x}}\,\boldsymbol{\xi}^{\mathrm{T}}\mathbf{n}(\mathbf{x})\,\mathrm{d}\sigma(\mathbf{x})
   	 \end{equation}
 	 with constants $C_0,\hdots, C_u>0$ and functions $p_0(\mathbf{x},\boldsymbol{\xi})\mathrel{\mathop:}=p(\mathbf{x})$, $p_m(\mathbf{x},\boldsymbol{\xi})\mathrel{\mathop:}=\frac{\boldsymbol{\xi}^{\mathrm{T}}}{\abs{\boldsymbol{\xi}}_2}\mathrm{grad}_\mathbf{x}[p_{m-1}](\mathbf{x},\boldsymbol{\xi})$ and the outer normal vector of the boundary $\partial T$ given by $\mathbf{n}(\mathbf{x})$.\\

 In the following lemma, we derive an explicit expression for the functions $p_m$ which gives a new representation of \cref{fourier_transformation_gauss} in polar coordinates.
    % Lemma: Fourier Transformation Gauß
    \begin{lemma}
    \label{lem:fourier_transformation_gauss}
	Let $T\in \mathrm{STAR^2}(\tau)$ and $p=p_u:\mathbb{R}^2\rightarrow \mathbb{R}$ be a bivariate polynomial of order $u$. Then there exist constants $C_0,\hdots, C_u>0$ such that the Fourier transform of the function $P_u=p\,\chi_T$ is of the form
   \begin{equation}\label{eq:T_fourier}
   	\mathcal{F}P_u\left(\rho\,\boldsymbol{\Theta}(\theta)\right)=\sum\limits_{m=0}^{u}\frac{C_m}{\rho^{m+1}}\int\limits_{\partial T}\,\partial_{\boldsymbol{\Theta}(\theta)}^m[p](\mathbf{x})\,\mathrm{e}^{-\mathrm{i}\,\rho\,\boldsymbol{\Theta}^{\mathrm{T}}(\theta)\mathbf{x}}\,\boldsymbol{\Theta}^{\mathrm{T}}(\theta)\,\mathbf{n}(\mathbf{x})\,\mathrm{d}\sigma(\mathbf{x}).
   \end{equation}
   \end{lemma}
    \begin{proof}
   	    By induction on the variable $m$ we show that
       	 \begin{equation*} 	p_m(\mathbf{x},\boldsymbol{\xi})=\abs{\boldsymbol{\xi}}_2^{-m}\sum\limits_{\abs{\mathbf{r}}_1=m}\binom{m}{\mathbf{r}}\,\boldsymbol{\xi}^{\mathbf{r}}\,\partial^{\mathbf{r}}[p](\mathbf{x}).
       	 \end{equation*}
   		For $m=0$ we have $p_0(\mathbf{x},\boldsymbol{\xi})=p$. Suppose there exists $m\in \mathbb{N}$ such that 
   		\begin{equation*}
			p_m(\mathbf{x},\boldsymbol{\xi})=\frac{\boldsymbol{\xi}^{\mathrm{T}}}{\abs{\boldsymbol{\xi}}_2}\mathrm{grad}_\mathbf{x}[p_{m-1}](\mathbf{x},\boldsymbol{\xi})=\abs{\boldsymbol{\xi}}_2^{-m}\sum\limits_{\abs{\mathbf{r}}_1=m}\binom{m}{\mathbf{r}}\,\boldsymbol{\xi}^{\mathbf{r}}\,\partial^{\mathbf{r}}p.
   		\end{equation*} 
 		It follows that
   	   	 \begin{align*}
   	   	 	\abs{\boldsymbol{\xi}}_2^{m+1}p_{m+1}(\mathbf{x},\boldsymbol{\xi})&=\boldsymbol{\xi}^{\mathrm{T}}\mathrm{grad}_{\mathbf{x}}\left[ \sum\limits_{\abs{\mathbf{r}}_1=m}\binom{m}{\mathbf{r}}\boldsymbol{\xi}^{\mathbf{r}}\,\partial^{\mathbf{r}}p \right]\\ 			&=\sum\limits_{r=0}^m\binom{m}{r}\,\xi_1^{r+1}\,\xi_2^{m-r}\,\partial^{(r+1,0)}\,\partial^{(0,m-r)}p+\sum\limits_{r=0}^m\binom{m}{r}\,\xi_1^{r}\,\xi_2^{m-r+1}\,\partial^{(r,0)}\,\partial^{(0,m-r+1)}p\\
   	 		&=\sum\limits_{\abs{\mathbf{r}}_1=m+1}\binom{m+1}{\mathbf{r}}\,\boldsymbol{\xi}^{\mathbf{r}}\,\partial^{\mathbf{r}}p.
   	   	 \end{align*}
   	 We use polar coordinates to verify
   		\begin{equation*}
   			\abs{\boldsymbol{\xi}}_2^{-m}\,\boldsymbol{\xi}^{\mathbf{r}}=(\cos\theta)^{r_1}\,(\sin\theta)^{r_2}=(\boldsymbol{\Theta}(\theta))^{\mathbf{r}}
   		\end{equation*}
 		if $\abs{\mathbf{r}}_1=m$ and obtain with \cref{eq:m_te_richtungsableitung_allgemein}
   		\begin{equation}\label{fourier_transformation_gauss2}
			p_m(\mathbf{x},\rho\,\boldsymbol{\Theta}(\theta))=\sum\limits_{\abs{\mathbf{r}}_1=m}\binom{m}{\mathbf{r}}\,(\boldsymbol{\Theta}(\theta))^{\mathbf{r}}\,\partial^{\mathbf{r}}[p](\mathbf{x})=\partial_{\boldsymbol{\Theta}(\theta)}^m[p](\mathbf{x}).
   		\end{equation}
   		We finish the proof by using polar coordinates for the variable $\boldsymbol{\xi}$ in \cref{fourier_transformation_gauss} and inserting \cref{fourier_transformation_gauss2}.
    \end{proof}
 			 Let $\boldsymbol{\gamma}:[0,2\pi)\rightarrow\partial T$ be a curve from \cref{eq:star1}. For $M\in \mathbb{N}$ let $a_0<a_1<\hdots<a_M$ be a partition of the interval $[0,2\pi)$ such that for each $x\in[a_k,a_{k+1}), k=0,\hdots,M-1,$ the curve $\boldsymbol{\gamma}$ can either be represented as a horizontal curve $(x,f(x))^{\mathrm{T}}$ or a vertical curve $(f(x),x)^{\mathrm{T}}$. If $\mathfrak{i}=\mathfrak{h}$, then $(f(x),x)^{\mathrm{T}}$ with $\abs{f'(x)}\leq 1$ is a vertical curve and $(x,f(x))^{\mathrm{T}}$ with $\abs{f'(x)}<1$ is a horizontal curve. Otherwise, if $\mathfrak{i}=\mathfrak{v}$, then $(f(x),x)^{\mathrm{T}}$ with $\abs{f'(x)}<1$ is a vertical curve and $(x,f(x))^{\mathrm{T}}$ with $\abs{f'(x)}\leq 1$ is a horizontal curve.\\
				
 		With the parametrization of the curve $\boldsymbol{\gamma}$ we can write the line integral \cref{eq:T_fourier} as
 		\begin{align*}
			\mathcal{F}P_u\left(\rho\,\boldsymbol{\Theta}(\theta)\right)&=\sum\limits_{m=0}^{u}\frac{C_m}{\rho^{m+1}}\int\limits_{0}^{2\pi}\partial_{\boldsymbol{\Theta}(\theta)}^m[p](\mathbf{x})\,\mathrm{e}^{-\mathrm{i}\,\rho\,\boldsymbol{\Theta}^{\mathrm{T}}(\theta)\,\boldsymbol{\gamma}(x)}\,\boldsymbol{\Theta}^{\mathrm{T}}(\theta)\,\mathbf{n}(\boldsymbol{\gamma}(x))\abs{\boldsymbol{\gamma}'(x)}_2\mathrm{d}x\\	
&=\sum\limits_{m=0}^{u}\frac{C_m}{\rho^{m+1}}\sum_{k=0}^{M-1}\int\limits_{a_k}^{a_{k+1}}p_{\theta}^m(\boldsymbol{\gamma}(x))\,\mathrm{e}^{-\mathrm{i}\,\rho\,\boldsymbol{\Theta}^{\mathrm{T}}(\theta)\,\boldsymbol{\gamma}(x)}\,\boldsymbol{\Theta}^{\mathrm{T}}(\theta)\,\boldsymbol{\beta}(x)\,\mathrm{d}x,
 		\end{align*}
 where $\boldsymbol{\beta}(x)\mathrel{\mathop:}=\mathbf{n}(\boldsymbol{\gamma}(x))\abs{\boldsymbol{\gamma}'(x)}_2$ and $p_{\theta}^m(\mathbf{x})\mathrel{\mathop:}=\partial_{\boldsymbol{\Theta}(\theta)}^m[p](\mathbf{x})$. With the help of polar coordinates, we transform the following integral into
     		\begin{align}
     			&\mathcal{F}^{-1}\left[ \mathcal{F}[P_u]\Psi^{(\mathfrak{i})}_{j,\ell} \right](2\pi\mathbf{y})\notag\\
     		&\qquad=\sum_{m=0}^{u}C_m\int\limits_{0}^{\infty}\int\limits_{0}^{2\pi}\int\limits_{\partial T}\Psi_{j,\ell}^{(\mathfrak{i})}\left(\rho\,\boldsymbol{\Theta}(\theta)\right)p_{\theta}^m(\mathbf{x})\rho^{-m}\mathrm{e}^{\mathrm{i} \rho \boldsymbol{\Theta}^{\mathrm{T}}(\theta)(2\pi\mathbf{y}-\mathbf{x})}\boldsymbol{\Theta}^{\mathrm{T}}(\theta)\,\mathbf{n}(\mathbf{x})\mathrm{d}\sigma\,\mathrm{d}\theta\,\mathrm{d}\rho\notag\\\label{eq:F_inv}
     		&\qquad=\sum_{m=0}^{u}C_m\sum_{k=0}^{M-1}\,I_k^{(\mathfrak{i})}(j,\ell,\mathbf{y},m)
     		\end{align}
 			with
     \begin{equation*}
		I_k^{(\mathfrak{i})}(j,\ell,\mathbf{y},m)\mathrel{\mathop:}=\int\limits_{0}^{\infty}\int\limits_{0}^{2\pi}\int\limits_{a_k}^{a_{k+1}}\Psi_{j,\ell}^{(\mathfrak{i})}\left(\rho\,\boldsymbol{\Theta}(\theta)\right)p_{\theta}^m(\boldsymbol{\gamma}(x))\rho^{-m}\,\mathrm{e}^{\mathrm{i} \rho \boldsymbol{\Theta}^{\mathrm{T}}(\theta)(2\pi\mathbf{y}-\boldsymbol{\gamma}(x))}\,\boldsymbol{\Theta}^{\mathrm{T}}(\theta)\,\boldsymbol{\beta}(x)\,\mathrm{d}x\,\mathrm{d}\theta\,\mathrm{d}\rho.
     \end{equation*}
 	From the assumption of \cref{thm:hauptresultat2} it follows that there exists $\varepsilon>0$ such that $U_\varepsilon(\mathbf{y})=\partial T\cap B_\varepsilon(2\pi\mathbf{y})\neq\emptyset$. We choose $k^*=k^*(\mathbf{y})$ with $0\leq k^*\leq M-1$ such that for $x\in[a_{k^*},a_{k^*+1})$ the neighborhood
 \begin{equation*}
 	U_\varepsilon(\mathbf{y})=\partial T\cap B_\varepsilon(2\pi\mathbf{y})
 \end{equation*}
 from \cref{eq:U_epsilon} can be represented by the curve $\boldsymbol{\gamma}(x)$ (see \cref{fig:skizze_rand}). 
 	\begin{figure}[t]
 		\subfloat{
 		{\includegraphics[width=.49\textwidth]{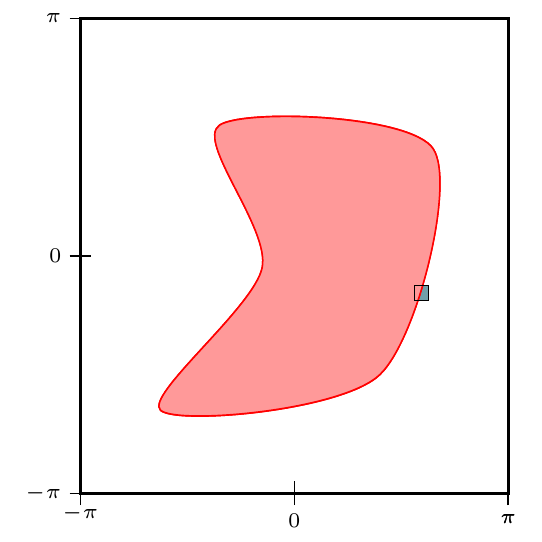}}
 		}\hspace{-0.5cm}
 		\subfloat{
 		{\includegraphics[width=.503\textwidth]{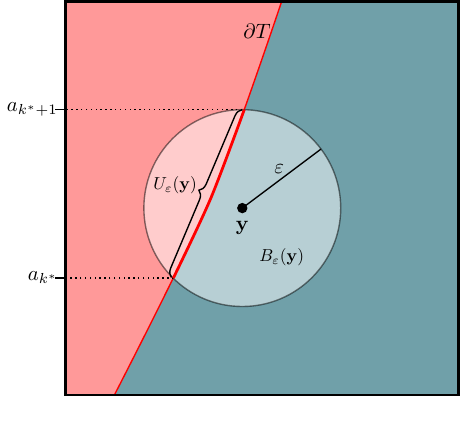}}
 		}
 	  	\caption[]{Left: Star-like set $T\in\mathrm{STAR}^2$ (red). Right: Zoom into the small window of the left picture to see the neighborhood $B_\varepsilon(\mathbf{y})$ around $\mathbf{y}\in \mathcal{P}(\mathbf{N}_{j,\ell}^{(\mathfrak{i})})$ and  $U_\varepsilon(\mathbf{y})$ on the boundary $\partial T$ with the interval $[a_{k^*},a_{k^*+1})$.}\label{fig:skizze_rand}
 	    \end{figure}
 The following lemma, called localization lemma, is important for the proof of \cref{thm:hauptresultat}. We adapt the main ideas of \cite[Lemma 4.1]{labate:detection_continuous} where a similar statement was shown for cone-adapted continuous shearlets.
        % Lemma: Lokalisierungslemma
        \begin{lemma}
        \label{lem:lokalisierungslemma}
     	For $\mathfrak{i}\in \left\lbrace \mathfrak{h},\mathfrak{v}\right\rbrace$ and $q\in \mathbb{N}$ let $\Psi^{(\mathfrak{i})}\in \mathcal{W}^{2q}$ be given. Then there exists a constant $C(m,q,p,\varepsilon_0)>0$ such that for all $k\neq k^*$ we have
   	\begin{equation*}
   		\bigl\lvert I_k^{(\mathfrak{i})}(j,\ell,\mathbf{y},m)\bigr\rvert\leq C(m,q,p,\varepsilon_0)\,2^{-j(q+m-1/2)}.
   	\end{equation*}
     \end{lemma}
     \begin{proof}
   	 We provide the proof only for $\mathfrak{i}=\mathfrak{h}$ and use the notation $I_k\mathrel{\mathop:}=I_k^{(\mathfrak{h})}(j,\ell,\mathbf{y},m)$. From \cite[Lemma 1]{schober:detection} we have
   	  \begin{equation*}
   	  	\mathrm{supp}\,\Psi_{j,\ell}^{(\mathfrak{h})}\left( 2^j\rho\,\boldsymbol{\Theta}(\theta) \right)\subset\left\lbrace(\rho,\theta)\in \mathbb{R}\times\left[-\frac{\pi}{2},\frac{\pi}{2}\right]:\frac{1}{3}<\abs{\rho}< 2,\,\theta_{j,\ell-2}^{(\mathfrak{h})}<\theta<\theta_{j,\ell+2}^{(\mathfrak{h})}\right\rbrace
   	  \end{equation*} 
   	  and the substitution $\rho=2^j\,\rho'$ leads to
   	  \begin{equation*}
   	  	I_k=2^{-j(m-1)}\int\limits_{\frac{1}{3}}^{2} \int\limits_{\theta_{j,\ell-2}^{(\mathfrak{h})}}^{\theta_{j,\ell+2}^{(\mathfrak{h})}}\int\limits_{a_k}^{a_{k+1}}\Psi_{j,\ell}^{(\mathfrak{h})}\left( 2^j\rho\,\boldsymbol{\Theta}(\theta) \right)\,p_{\theta}^m(\boldsymbol{\gamma}(x))\,\rho^{-m}\,\mathrm{e}^{\mathrm{i}2^j\rho\boldsymbol{\Theta}^{\mathrm{T}}(\theta)(2\pi \mathbf{y}-\boldsymbol{\gamma}(x))}\boldsymbol{\Theta}^{\mathrm{T}}(\theta)\boldsymbol{\beta}(x)\,\mathrm{d}x\,\mathrm{d}\theta\, \mathrm{d}\rho.
   	  \end{equation*}
 	 We consider the sets
   	  \begin{equation*}
   	  	M_1\mathrel{\mathop:}=M_1(K)\mathrel{\mathop:}=\left\lbrace\theta\in \left(\theta_{j,\ell-2}^{(\mathfrak{h})},\theta_{j,\ell+2}^{(\mathfrak{h})}\right):\frac{\bigl\lvert\boldsymbol{\Theta}^{\mathrm{T}}(\theta)(2\pi\mathbf{y}-\boldsymbol{\gamma}(x))\bigr\rvert}{\abs{2\pi \mathbf{y}-\boldsymbol{\gamma}(x)}_2}\geq K\right\rbrace
 		\end{equation*}
 		and 
 		\begin{equation*}
   		 M_2\mathrel{\mathop:}=\left(\theta_{j,\ell-2}^{(\mathfrak{h})},\theta_{j,\ell+2}^{(\mathfrak{h})}\right)\setminus M_1,
     	  \end{equation*}
 		where $K=K(\varepsilon_0)>0$ is chosen such that $M_1=\left(\theta_{j,\ell-2}^{(\mathfrak{h})},\theta_{j,\ell+2}^{(\mathfrak{h})}\right)$ for all $x\in[a_k,a_{k+1}]$ with $\abs{2\pi \mathbf{y}-\boldsymbol{\gamma}(x)}_2<c(\varepsilon_0)$. \\
  	  
 	  We can use these sets to split the integral into $I_k=I_{k,1}+I_{k,2}$, where
   	  \begin{align*}
   	  	I_{k,i}&\mathrel{\mathop:}=2^{-j(m-1)}\int\limits_{\frac{1}{3}}^{2} \int\limits_{M_i}\int\limits_{a_k}^{a_{k+1}}\Psi_{j,\ell}^{(\mathfrak{h})}\left( 2^j\rho\,\boldsymbol{\Theta}(\theta) \right)\,p_{\theta}^m(\boldsymbol{\gamma}(x))\,\rho^{-m}\,\mathrm{e}^{\mathrm{i}2^j\rho\boldsymbol{\Theta}^{\mathrm{T}}(\theta)(2\pi \mathbf{y}-\boldsymbol{\gamma}(x))}\\
 		&\qquad\times\boldsymbol{\Theta}^{\mathrm{T}}(\theta)\,\boldsymbol{\beta}(x)\,\mathrm{d}x\,\mathrm{d}\theta\, \mathrm{d}\rho
   	  \end{align*}
   	 for $i\in \left\lbrace 1,2 \right\rbrace$ and investigate these integrals separately.
 	 \begin{itemize}
		 \item[i)] By Fubini's theorem, we can change the order of integration in $I_{k,1}$ to obtain
 	  \begin{equation*}
 	  		I_{k,1}=2^{-j(m-1)} \int\limits_{M_1}\int\limits_{a_k}^{a_{k+1}}J(x,\theta)\,p_{\theta}^m(\boldsymbol{\gamma}(x))\,\boldsymbol{\Theta}^{\mathrm{T}}(\theta)\,\boldsymbol{\beta}(x)\,\mathrm{d}x\,\mathrm{d}\theta
 	  \end{equation*} 
 	 with
   	  \begin{equation*} 	
 		  J(x,\theta)\mathrel{\mathop:}=\int\limits_{\frac{1}{3}}^{2}\Psi_{j,\ell}^{(\mathfrak{h})}\left( 2^j\rho\,\boldsymbol{\Theta}(\theta) \right)\,\rho^{-m}\,\mathrm{e}^{\mathrm{i}2^j\rho\boldsymbol{\Theta}^{\mathrm{T}}(\theta)(2\pi \mathbf{y}-\boldsymbol{\gamma}(x))}\,\mathrm{d}\rho,\quad x\in[a_k,a_{k+1}),\quad\theta\in M_1.
   	  \end{equation*}
   	  For $k=0,\hdots,M-1$ with $k\neq k^*$ and $x\in[a_k,a_{k+1})$ we have $\boldsymbol{\gamma}(x)\in U_{\varepsilon}^{\mathrm{c}}(\mathbf{y})$ or equivalently 
	  \begin{equation}\label{eq:lokalisierungslemma0}
	  	\abs{2\pi \mathbf{y}-\boldsymbol{\gamma}(x)}_2\geq\varepsilon=\varepsilon_0\,2^{-j/2}.
	  \end{equation}
 	 With
 	  \begin{equation*}
 	  	\abs{\frac{\partial^s}{\partial\rho^s}\left[\rho^{-m}\right]}=\frac{(m+s)!}{(m-1)!}\abs{\rho}^{-(m+s)},
 	  \end{equation*}
 	  the Leibniz rule and \cref{lem:partielle_ableitung_psi:polar} we obtain
 	  \begin{align}
 	\abs{\frac{\partial^{2q}}{\partial\rho^{2q}}\left[\Psi_{j,\ell}^{(\mathfrak{h})}\left( 2^j\rho\,\boldsymbol{\Theta}(\theta) \right)\,\rho^{-m}\right]}&\leq\sum_{s=0}^{2q}\binom{2q}{s}\abs{\frac{\partial^s}{\partial\rho^s}\left[\Psi_{j,\ell}^{(\mathfrak{h})}\left( 2^j\rho\,\boldsymbol{\Theta}(\theta) \right)\right]}\abs{\frac{\partial^{2q-s}}{\partial\rho^{2q-s}}\left[\rho^{-m}\right]}\notag\\
 	\label{eq:lokalisierungslemma1}
 	&\leq C_2(q,m).
 	  \end{align}
	  Since $\theta\in M_1$, we can use \cref{eq:lokalisierungslemma0} and \cref{eq:lokalisierungslemma1} for the integral $J(x,\theta)$ and $2q$-times integration by parts with respect to the variable $\rho$ to obtain
     	  \begin{align*}
     	\bigl\lvert J(x,\theta)\bigr\rvert&\leq\Bigl\lvert2^j\,\boldsymbol{\Theta}^{\mathrm{T}}(\theta)(2\pi \mathbf{y}-\boldsymbol{\gamma}(x))\Bigr\rvert^{-2q}\int\limits_{\frac{1}{3}}^{2}\abs{\frac{\partial^{2q}}{\partial\rho^{2q}}\left[\Psi_{j,\ell}^{(\mathfrak{h})}\left( 2^j\rho\,\boldsymbol{\Theta}(\theta) \right)\,\rho^{-m}\right]}\,\mathrm{d}\rho\\
		&\leq C_3(q,m,\varepsilon_0)\,2^{-jq}.
     	  \end{align*}
   	  The estimate \cref{proof:int_FT_Q1_2} implies $\abs{M_1}\leq C\,2^{-j/2}$ and we can bound the integral $I_{k,1}$ from above by
   	  \begin{align*}
   	  	\bigl\lvert I_{k,1}\bigr\rvert&\leq2^{-j(m-1)} \int\limits_{M_1}\int\limits_{a_k}^{a_{k+1}}\bigl\lvert J(x,\theta)\bigr\rvert\abs{p_{\theta}^m(\boldsymbol{\gamma}(x))\,\boldsymbol{\Theta}^{\mathrm{T}}(\theta)\,\boldsymbol{\beta}(x)}\,\mathrm{d}x\,\mathrm{d}\theta\\
		&\leq C_4(q,m,p,\varepsilon_0)\,2^{-j(q+m-1/2)}.
   	  \end{align*}
	  \item[ii)] For the integral $I_{k,2}$ we follow a similar approach, but this time with respect to the variable $\theta$. For $x\in[a_k,a_{k+1}]$ with $\abs{2\pi \mathbf{y}-\boldsymbol{\gamma}(x)}_2<c(\varepsilon_0)$, by the choice of $K$ in $M_1$, we have that $M_2=\emptyset$ and thus $I_{k,2}=0$. If on the other hand $M_2\neq\emptyset$, we have $\boldsymbol{\Theta}^{\mathrm{T}}(\theta)\boldsymbol{\Theta}'(\theta)=0$ for $\theta\in M_2$ and the inequality
	    	  \begin{equation}\label{eq:lokalisierungslemma2}
	    	  	\Bigl\lvert(2\pi\mathbf{y}-\boldsymbol{\gamma}(x))^{\mathrm{T}}\boldsymbol{\Theta}'(\theta)\Bigr\rvert\geq c(\varepsilon_0)
	    	  \end{equation}
	  	  is fulfilled. We can write the integral $I_{k,2}$ as
 	 \begin{equation*}
 	 	I_{k,2}=2^{-j(m-1)} \int\limits_{\frac{1}{3}}^{2}\int\limits_{a_k}^{a_{k+1}}K(x,\rho)\,\rho^{-m}\,\mathrm{d}x\,\mathrm{d}\rho,
 	 \end{equation*}
 	 where
   	 \begin{equation*} 	K(x,\rho)\mathrel{\mathop:}=\int\limits_{M_2}\Psi_{j,\ell}^{(\mathfrak{h})}\left( 2^j\rho\,\boldsymbol{\Theta}(\theta) \right)\,p_{\theta}^m(\boldsymbol{\gamma}(x))\,\mathrm{e}^{\mathrm{i}2^j\rho\boldsymbol{\Theta}^{\mathrm{T}}(\theta)(2\pi \mathbf{y}-\boldsymbol{\gamma}(x))}\,\boldsymbol{\Theta}^{\mathrm{T}}(\theta)\,\boldsymbol{\beta}(x)\,\mathrm{d}\theta.
   	 \end{equation*}
 	 The curve \cref{eq:star1} is of the form $\boldsymbol{\gamma}(x)=\mathbf{x}_0+r(x)(\cos x,\,\sin x)^{\mathrm{T}}$. Thus, we have
 	\begin{align*}
 	 	\boldsymbol{\beta}(x)&=\mathbf{n}(\boldsymbol{\gamma}(x))\abs{\boldsymbol{\gamma}'(x)}_2\\
		&=\left( r(x)\cos x+r'(x)\sin x,\,r(x)\sin x-r'(x)\cos x\right)^{\mathrm{T}}\sqrt{r(x)^2+r'(x)^2}.
 	 \end{align*}
 	 From this equality we obtain
 	 \begin{equation*}
 	 	\boldsymbol{\Theta}^{\mathrm{T}}(\theta)\,\boldsymbol{\beta}(x)=\left( r(x)\cos\left( \theta-x \right)+r'(x)\sin\left( \theta-x \right) \right)\sqrt{r(x)^2+r'(x)^2}
 	 \end{equation*}
 	leading to
 	 \begin{equation*}
 	 	\abs{\frac{\partial^s}{\partial\theta^s}\left[ \boldsymbol{\Theta}^{\mathrm{T}}(\theta)\,\boldsymbol{\beta}(x) \right]}\leq C_5.
 	 \end{equation*}
 	 With the same ideas which led to \cref{eq:lokalisierungslemma1} we can estimate
 	 \begin{align}
 	 	\abs{\frac{\partial^{2q}}{\partial\theta^{2q}}\left[\Psi_{j,\ell}^{(\mathfrak{h})}\left( 2^j\rho\boldsymbol{\Theta}(\theta) \right)\boldsymbol{\Theta}^{\mathrm{T}}(\theta)\boldsymbol{\beta}(x)\right]}&\leq\sum_{s=0}^{2q}\binom{2q}{s}C_6(s,m,p)\,2^{js/2}\notag\\\label{eq:lokalisierungslemma3}
 		&\leq C_7(q,m,p)\,2^{jq}.
 	 \end{align}
	 Finally, we perform $2q$-times integration by parts with respect to the variable $\theta$ and use \cref{eq:lokalisierungslemma2} and \cref{eq:lokalisierungslemma3} to obtain the estimate
	   	 \begin{align*}
	 &\bigl\lvert K(x,\rho)\bigr\rvert \\
	 &\qquad\leq \int\limits_{M_2}\Bigl\lvert2^j\rho\,(2\pi\mathbf{y}-\boldsymbol{\gamma}(x))^{\mathrm{T}}\boldsymbol{\Theta}'(\theta)\Bigl\lvert ^{-2q}\abs{\frac{\partial^{2q}}{\partial\theta^{2q}}\left[\Psi_{j,\ell}^{(\mathfrak{h})}\left( 2^j\rho\boldsymbol{\Theta}(\theta) \right)p_{\theta}^m(\boldsymbol{\gamma}(x))\boldsymbol{\Theta}^{\mathrm{T}}(\theta)\boldsymbol{\beta}(x)\right]}\mathrm{d}\theta\\
	 &\qquad\leq C_8(q,m,p,\varepsilon_0)\,2^{-j(q+1/2)}.
	   	 \end{align*}
   	 Similar to the estimate for $I_{k,2}$, we get
  	 \begin{equation*}
  	 	\bigl\lvert I_{k,2}\bigr\rvert\leq2^{-j(m-1)}\int\limits_{\frac{1}{3}}^{2} \int\limits_{a_k}^{a_{k+1}}\bigl\lvert K(x,\rho)\bigr\rvert \abs{\rho^{-m}}\,\mathrm{d}x\,\mathrm{d}\rho\leq C_9(q,m,p,\varepsilon_0)\,2^{-j(q+m-1/2)}
  	 \end{equation*}
   	 and the proof is finished.
	 \end{itemize}
	 \vspace{-0.7cm}
     \end{proof}
	The set $\mathcal{M}^{(\mathfrak{h})}\subset\lbrace 0,\hdots,M-1\rbrace$ contains all indices such that for $x\in[a_k,a_{k+1})$ with $k\in\mathcal{M}^{(\mathfrak{h})}$ the curve $\boldsymbol{\gamma}(x)$ is horizontal and $\mathcal{M}^{(\mathfrak{v})}\subset\lbrace 0,\hdots,M-1\rbrace$ includes all indices such that for $x\in[a_k,a_{k+1})$ with $k\in\mathcal{M}^{(\mathfrak{v})}$ the curve $\boldsymbol{\gamma}(x)$ is vertical. Obviously, we have $\mathcal{M}^{(\mathfrak{h})}\cup\mathcal{M}^{(\mathfrak{v})}=\lbrace 1,\hdots,M\rbrace$. A similar version of the following lemma was proven in \cite[Lemma 5.10]{schober:detection}, based on the ideas from \cite[Section 3.1]{labate:detection}. We omit the proof here since it requires only a slight adjustment of the proof of \cite[Lemma 5.10]{schober:detection}.
 % Lemma: Orientation Lemma
     \begin{lemma}\label{lem:orientation_lemma} 
     	For $\mathfrak{i}\in\lbrace\mathfrak{h},\mathfrak{v}\rbrace$ and $q\in \mathbb{N}$ let $\Psi^{(\mathfrak{i})}\in \mathcal{W}^q$ be a window function. Then for any $N\in \mathbb{N}$ there exists a constant $C(m,N,p)>0$ such that for all $k\in \mathcal{M}^{(\mathfrak{i})}$ we have
     	\begin{equation*}
 			\bigl\lvert I_k^{(\mathfrak{i})}(j,\ell,\mathbf{y},m)\bigr\rvert\leq C(m,N,p)\,2^{-j(N+m-1/2)}.
     	\end{equation*}
     \end{lemma}
	In the case of continuous shearlets, the following lemma was already established in \cite[Lemma 4.3]{labate:smooth}.
 	  \begin{lemma}\label{lem:P_L}
 		  Let $\mathfrak{f}=f\,\chi_T\in \mathcal{E}^{u+1}(\tau)$ and $T_uf(\mathbf{x};\,2\pi\mathbf{y})$ be the bivariate Taylor approximation of $f$ with order $u$ around the point $2\pi\mathbf{y}$ and let $P_{u,f,\mathbf{y}}(\mathbf{x})\mathrel{\mathop:}=T_uf(\mathbf{x};\,2\pi\mathbf{y})\,\chi_T(\mathbf{x})$. Moreover, for $\mathfrak{i}\in\lbrace\mathfrak{h},\mathfrak{v}\rbrace$ and $2q\geq u\in \mathbb{N}$ let $\Psi^{(\mathfrak{i})}\in \mathcal{W}^{2q}$ be a window function. Then there is a constant $C(\mathfrak{f},q)>0$ such that
 	  	\begin{equation*}
 	  		\abs{\left\langle \mathfrak{f}^{2\pi}-P_{u,f,\mathbf{y}}^{2\pi},\psi_{j,\ell,\mathbf{y}}^{(\mathfrak{i})} \right\rangle_2}\leq C(\mathfrak{f},q)\,2^{-j(u-1)/4}.
 	  	\end{equation*}
 	  \end{lemma}
 	  \begin{proof}
 	  	Again, we present the proof only for $\mathfrak{i}=\mathfrak{h}$. With $\delta=2^{-j/4}$ we get
 	  	\begin{align*}
 			\abs{\left\langle \mathfrak{f}^{2\pi}-P_{u,f,\mathbf{y}}^{2\pi},\psi_{j,\ell,\mathbf{y}}^{(\mathfrak{h})} \right\rangle_2}&\leq\int\limits_{\mathbb{T}^2}\psi^{(\mathfrak{h})}_{j,\ell,\mathbf{y}}(\mathbf{x})\,\chi_{T}(\mathbf{x})\,\bigl\lvert f(\mathbf{x})-P_{u,f,\mathbf{y}}(\mathbf{x})\bigr\rvert\mathrm{d}\mathbf{x}\\
 			&=\Biggl(\;\int\limits_{B_{\delta}(2\pi \mathbf{y})}+\int\limits_{B^{\mathrm{c}}_{\delta}(2\pi \mathbf{y})}\Biggr)\psi^{(\mathfrak{h})}_{j,\ell,\mathbf{y}}(\mathbf{x})\,\chi_{T}(\mathbf{x})\,\bigl\lvert f(\mathbf{x})-P_{u,f,\mathbf{y}}(\mathbf{x})\bigr\rvert\mathrm{d}\mathbf{x}
 	  	\end{align*}
 	  	and write the last line as $\mathcal{I}_1+\mathcal{I}_2$. The approximation property of Taylor polynomials of order $u$ leads to
 	  	\begin{equation*}
 	  		\bigl\lvert f(\mathbf{x})-P_{u,f,\mathbf{y}}(\mathbf{x})\bigr\rvert\leq C\,2^{-j(u+1)/4}
 	  	\end{equation*}
 	  	for $\mathbf{x}\in B_{\delta}(2\pi \mathbf{y})$. The result in \cite[Lemma 9]{schober:detection} implies $\abs{\psi_{j,\ell,\mathbf{y}}^{(\mathfrak{h})}(\mathbf{x})}\leq C(q)\,2^{3j/4}$. Thus, we can bound $\mathcal{I}_1$ by
 	  	\begin{equation*}
 	  	 \abs{\mathcal{I}_1}\leq C(q)\,2^{3j/4}\int\limits_{ B_{\delta}(2\pi \mathbf{y})\cap T}\bigl\lvert f(\mathbf{x})-P_{u,f,\mathbf{y}}(\mathbf{x})\bigr\rvert\mathrm{d}\mathbf{x}\leq C(q)\, 2^{-j\left((u+2)/4-3/4\right)}=C(q)\,2^{-j(u-1)/4}.
 	  	\end{equation*}
 	 We again use \cite[Lemma 9]{schober:detection} but this time for the decay term to arrive at
 	  	\begin{align*}
 	  	 \abs{\mathcal{I}_2}&\leq C(q)\,2^{-j(q-3/4)}\int_{0}^{2\pi}\int_{2^{-\frac{j}{4}}}^{\infty}\rho^{1-2q}\,\mathrm{d}\rho\,\mathrm{d}\theta\\
 		 &\leq C_2(q)\,2^{-j(q-3/4)}\,2^{-j(1-q)/2}=C_3(q)\ 2^{-j(2q-1)/4}
 	  	\end{align*}
 	  	and the proof is finished since $2q\geq u$.
 	\end{proof}
     		 
% Section localization_lemmata (end)

% Section Proof of Theorem 3.2
{\section{Proof of Theorem 3.2} 
\label{sec:proof_of_theorem_3_2}

 Let $T\in \mathrm{STAR}^2(\tau)$ and $\mathfrak{f}\mathrel{\mathop:}=f\,\chi_{T}\in \mathcal{E}^{u+1}(\tau)$ be given. Moreover, let $p\mathrel{\mathop:}=p_u\mathrel{\mathop:}=T_uf(\mathbf{x};\,2\pi\mathbf{y})$ be the bivariate Taylor polynomial of $f$ with order $u>4$ around the point $2\pi\mathbf{y}$. We consider functions $P_u=p\,\chi_T$ and denote its $2\pi$-periodization by $P_u^{2\pi}$.\\
 
 In the first part of the proof we show
 \begin{equation}\label{proof:hauptresultat2_0}
 	\abs{\left\langle P_u^{2\pi},\psi^{(\mathfrak{i})}_{j,\ell,\mathbf{y}} \right\rangle_2}\geq C(u,n,q,\varepsilon_0,T)\,2^{-j(3/4+n)}.
 \end{equation}
 Similar to \cref{eq:beweis_der_oberen_schranke1} in the proof of \cref{thm:hauptresultat}, we use Parseval's identity and the Poisson summation formula to get
  \begin{equation*}
  	\left\langle P_u^{2\pi},\psi^{(\mathfrak{i})}_{j,\ell,\mathbf{y}}\right\rangle_2=2^{-3j/4}\sum_{\mathbf{n}\in \mathbb{Z}^2}\mathcal{F}^{-1}\left[\mathcal{F}[P_u]\Psi^{(\mathfrak{i})}_{j,\ell} \right]\Bigl(2\pi(\mathbf{y}+\mathbf{n})\Bigr)=2^{-3j/4}\sum_{\mathbf{n}\in \mathbb{Z}^2}S(\mathbf{n}),
  \end{equation*}
where
  \begin{equation*}
 	S(\mathbf{n})\mathrel{\mathop:}=\mathcal{F}^{-1}\left[\mathcal{F}[P_u]\Psi^{(\mathfrak{i})}_{j,\ell} \right]\Bigl(2\pi(\mathbf{y}+\mathbf{n})\Bigr).
  \end{equation*}
  i) 
  
  First, we show
  \begin{equation}\label{proof:hauptresultat2_1}
  	2^{-3j/4}\sum_{\mathbf{n}\in \mathbb{Z}^2\setminus\{\mathbf{0}\}}\abs{S(\mathbf{n})}\leq C(u,n,q)\,2^{-j(q+n+1/4)}.
  \end{equation}
  We consider the decomposition of $P_u$ on dyadic squares $Q\in \mathcal{Q}_j$ and define $P_{u,Q}\mathrel{\mathop:}=P_u\,\phi_Q$ to get
  \begin{equation*}
  	P_u=\sum_{Q\in \mathcal{Q}_j^0}P_{u,Q}+\sum_{Q\in \mathcal{Q}_j^1}P_{u,Q}.
  \end{equation*}
  We repeat the steps which led to \cref{eq:beweis_der_oberen_schranke4} and obtain
 \begin{equation*}
 	2^{-3j/4}\sum_{\mathbf{n}\in \mathbb{Z}^2\setminus\{\mathbf{0}\}}\abs{\mathcal{F}^{-1}\left[\mathcal{F}[P_{u,Q}]\Psi^{(\mathfrak{i})}_{j,\ell} \right]\Bigl(2\pi(\mathbf{y}+\mathbf{n})\Bigr)}\leq C(q)\, 2^{-jq}\norm{L^q\left[\mathcal{F}[P_{L,Q}]\,\Psi^{(\mathfrak{i})}_{j,\ell}\right] }_{\mathrm{supp}\,\Psi_{j,\ell}^{(\mathfrak{i})},2}.
 \end{equation*}
 With the linearity of the Fourier transform and the estimate of the absolutely convergent series in the last line we can write 
  \begin{align*}
  	2^{-3j/4}\sum_{\mathbf{n}\in \mathbb{Z}^2\setminus\{\mathbf{0}\}}\abs{S(\mathbf{n})}&\leq2^{-3j/4}\sum_{\mathbf{n}\in \mathbb{Z}^2\setminus\{\mathbf{0}\}}\left( \sum_{Q\in \mathcal{Q}_j^0}+\sum_{Q\in \mathcal{Q}_j^1} \right) \abs{\mathcal{F}^{-1}\left[\mathcal{F}[P_{u,Q}]\Psi^{(\mathfrak{i})}_{j,\ell} \right]\Bigl(2\pi(\mathbf{y}+\mathbf{n})\Bigr)}\\
 	&=\left( \sum_{Q\in \mathcal{Q}_j^0}+\sum_{Q\in \mathcal{Q}_j^1} \right)2^{-3j/4}\sum_{\mathbf{n}\in \mathbb{Z}^2\setminus\{\mathbf{0}\}} \abs{\mathcal{F}^{-1}\left[\mathcal{F}[P_{u,Q}]\Psi^{(\mathfrak{i})}_{j,\ell} \right]\Bigl(2\pi(\mathbf{y}+\mathbf{n})\Bigr)}\\
 	&\leq C(q)\, 2^{-jq}\left( \sum_{Q\in \mathcal{Q}_j^0}+\sum_{Q\in \mathcal{Q}_j^1} \right)\norm{L^q\left[\mathcal{F}[P_{u,Q}]\,\Psi^{(\mathfrak{i})}_{j,\ell}\right] }_{\mathrm{supp}\,\Psi_{j,\ell}^{(\mathfrak{i})},2}.
  \end{align*}
  Next, we use \cref{lem:norm_Lq} and the estimates from \cref{eq:maechtigkeit_Q_j} and obtain \cref{proof:hauptresultat2_1} since
  \begin{align*}
  	2^{-3j/4}\sum_{\mathbf{n}\in \mathbb{Z}^2\setminus\{\mathbf{0}\}}\abs{S(\mathbf{n})}&\leq C(q)\, 2^{-jq}\left( C_1(u,q)\,2^j\,2^{-j(q+u+3/2)} + C_2(n,q)\, 2^{j/2}\,2^{3/4+n}\right)\\
 	&\leq C_3(u,n,q)\,2^{-j(q+n+1/4)}.
  \end{align*}
ii)

In the following, we show
    				\begin{equation}\label{eq:lower_bound}
    					\abs{S(\mathbf{0})}\geq C_4(n,q,\varepsilon_0,T)\,2^{-jn}.
    				\end{equation}
 				Assume, the last estimate is true. Then, for sufficiently large $q\in \mathbb{N}$ with \cref{proof:hauptresultat2_1} and the inverse triangle inequality we obtain
	   		\begin{equation*}
	   			\abs{\left\langle P_u^{2\pi},\psi^{(\mathfrak{i})}_{j,\ell,\mathbf{y}} \right\rangle_2}\geq 2^{-3j/4}\left(  \abs{S(\mathbf{0})}-\sum_{\mathbf{n}\in \mathbb{Z}^2\setminus\{\mathbf{0}\}}\abs{S(\mathbf{n})}\right)\geq C_5(u,n,q,\varepsilon_0,T)\,2^{-j(3/4+n)}
	   		\end{equation*}
			and therefore \cref{proof:hauptresultat2_0}.\\
	  
    		We now start with the proof of \cref{eq:lower_bound}. For this reason, we recall the representation 
 		  \begin{equation*}
				S(\mathbf{0})=\mathcal{F}^{-1}\left[\mathcal{F}[P_u]\Psi^{(\mathfrak{i})}_{j,\ell}\right](2\pi\mathbf{y})=\sum_{m=0}^{u}C_m\sum_{k=0}^{M-1}\,I_k^{(\mathfrak{i})}(j,\ell,\mathbf{y},m)
 		 \end{equation*}
		 from \cref{eq:F_inv} in polar coordinates with
 		    \begin{align}
				I_k^{(\mathfrak{i})}(j,\ell,\mathbf{y},m)&=\int\limits_{0}^{\infty}\int\limits_{0}^{2\pi}\int\limits_{a_k}^{a_{k+1}}\Psi_{j,\ell}^{(\mathfrak{i})}\left(\rho\,\boldsymbol{\Theta}(\theta)\right)p_{\theta}^m(\boldsymbol{\gamma}(x))\rho^{-m}\,\mathrm{e}^{\mathrm{i} \rho \boldsymbol{\Theta}^{\mathrm{T}}(\theta)(2\pi\mathbf{y}-\boldsymbol{\gamma}(x))}\times\boldsymbol{\Theta}^{\mathrm{T}}(\theta)\,\boldsymbol{\beta}(x)\,\mathrm{d}x\,\mathrm{d}\theta\,\mathrm{d}\rho\notag\\\label{proof:upper_bound1}
 		&=2^{-j(m-1)}\int\limits_{0}^{\infty}\int\limits_{0}^{2\pi}\Psi_{j,\ell}^{(\mathfrak{i})}\left( 2^j\rho\,\boldsymbol{\Theta}(\theta) \right)\,\rho^{-m}\,\mathrm{e}^{2\pi\mathrm{i}2^j\rho\,\boldsymbol{\Theta}^{\mathrm{T}}(\theta)\mathbf{y}}\,G_k(\rho,\theta,m)\,\mathrm{d}\theta\, \mathrm{d}\rho
 		    \end{align}
 			  and
 			  \begin{equation}\label{eq:L_k}
				G_k(\rho,\theta,m)\mathrel{\mathop:}=\int\limits_{a_k}^{a_{k+1}}p_{\theta}^m(\boldsymbol{\gamma}(x))\,\mathrm{e}^{-\mathrm{i}2^j\rho\boldsymbol{\Theta}^{\mathrm{T}}(\theta)\boldsymbol{\gamma}(x)}\,\boldsymbol{\Theta}^{\mathrm{T}}(\theta)\,\boldsymbol{\beta}(x)\,\mathrm{d}x.
 			  \end{equation}
    		We consider only $\mathfrak{i}=\mathfrak{h}$ since the case $\mathfrak{i}=\mathfrak{v}$ is similar. First, from \cref{lem:orientation_lemma} and the inverse triangle inequality it follows that
    		\begin{equation*}
    			\abs{\mathcal{F}^{-1}\left[\mathcal{F}[P_u]\Psi^{(\mathfrak{h})}_{j,\ell}\right](2\pi\mathbf{y})}\geq\abs{\sum_{m=0}^{u}C_m\sum_{k\in \mathcal{M}^{(\mathfrak{v})}}\,I_k^{(\mathfrak{h})}(j,\ell,\mathbf{y},m)}-\abs{\mathcal{M}^{(\mathfrak{h})}}C_6(m,N,p)\,2^{-j(N+m-1/2)}
    		\end{equation*}
 		and the last term is negligible for sufficiently large $N\in \mathbb{N}$. By assumption of the theorem, the set \cref{eq:U_epsilon} is nonempty and there exists $k^*=k^*(\mathbf{y}),\,0\leq k^*\leq M-1,$ such that $U_\varepsilon(\mathbf{y})$ can be represented by a vertical curve $\boldsymbol{\gamma}(x)=(t_{k^*}(x),x)^{\mathrm{T}}$ for $x\in[a_{k^*},a_{k^*+1})$. Thus, we can use \cref{lem:lokalisierungslemma} to bound the expression from above by
    		\begin{align*}
    			\abs{\sum_{m=0}^{u}C_m\sum_{k\in \mathcal{M}^{(\mathfrak{v})}}\,I_k^{(\mathfrak{h})}(j,\ell,\mathbf{y},m)}\geq\abs{\sum\limits_{m=0}^{u}C_m\,I_{k^*}^{(\mathfrak{h})}(j,\ell,\mathbf{y},m)}-C_7(m,p,q,\varepsilon_0)\,2^{-j(q+m-1/2)}.
    		\end{align*}
			Therefore, the desired estimate \cref{eq:lower_bound} is equivalent to find a constant $C_8(n,T)$ such that
    		\begin{equation}\label{eq:haupttheorem2_beh}
    			\abs{I_{k^*}^{(\mathfrak{h})}(j,\ell,\mathbf{y},m)}\geq C_8(n,T)\,2^{-jn}.
    		\end{equation}
 			For $k=k^*$ we split up the integral $I_{k^*}^{(\mathfrak{h})}(j,\ell,\mathbf{y},m)$ from \cref{proof:upper_bound1} into
 	  	  \begin{align*}
 	  	  	I_{k^*}^{(\mathfrak{h})}(j,\ell,\mathbf{y},m)&=2^{-j(m-1)}\int\limits_{0}^{\infty} \biggl(\int\limits_{-\frac{\pi}{2}}^{\frac{\pi}{2}}+\int\limits_{\frac{\pi}{2}}^{\frac{3\pi}{2}}\biggr)\Psi_{j,\ell}^{(\mathfrak{h})}\left( 2^j\rho\,\boldsymbol{\Theta}(\theta) \right)\,\rho^{-m}\,\mathrm{e}^{2\pi\mathrm{i}2^j\rho\,\boldsymbol{\Theta}^{\mathrm{T}}(\theta)\mathbf{y}}\mathcal{L}_{k^*}(\rho,\theta,m)\,\mathrm{d}\theta\, \mathrm{d}\rho\\
 			&=\mathrel{\mathop:}I_{k^*,1}^{(\mathfrak{h})}(j,\ell,\mathbf{y},m)+I_{k^*,2}^{(\mathfrak{h})}(j,\ell,\mathbf{y},m).
 	  	  \end{align*}
 		For convenience, we write $I_{1}\mathrel{\mathop:}=I_{k^*,1}^{(\mathfrak{h})}(j,\ell,\mathbf{y},m)$ and $I_{2}\mathrel{\mathop:}=I_{k^*,2}^{(\mathfrak{h})}(j,\ell,\mathbf{y},m)$ for the rest of the proof. We follow the ideas from \cite[p. 34]{schober:detection} and use the symmetry properties of the admissible functions $\widetilde{g}$ and $g$ to obtain
		\begin{equation}\label{eq:I1_I2}
			I_{k^*}=2\,\mathrm{i}\,\mathrm{Im}(I_{1})=2\,\mathrm{i}\,\mathrm{Im}(I_{2}).
		\end{equation}
		The vertical curve $\boldsymbol{\gamma}(x)$ is parametrized by $(t_{k^*}(x),x)^{\mathrm{T}}$ for $x\in[a_{k^*},a_{k^*+1})$. For the point $\mathbf{x}_0=(t_{k^*}(x_0),x_0)^{\mathrm{T}}\in U_{\varepsilon}(\mathbf{y})$ we have that $x_0\in[a_{k^*},a_{k^*+1})$ and therefore $\abs{x-x_0}<\varepsilon=\varepsilon_0\,2^{-j/2}$. We write the function $t_{k^*}(x)$ locally as
    		\begin{equation*}
    			t_{k^*}(x)=t_{k^*}(x_0)+B(x-x_0)+A(x-x_0)^2+r(x-x_0),
    		\end{equation*} 
 		where $r(x-x_0)=\mathcal{O}\left((x-x_0)^3\right)$ and in the case $\mathfrak{i}=\mathfrak{h}$ we have $B=t_{k^*}'(x_0)\in[-1,1]$. In the following, we assume $A\mathrel{\mathop:}=\frac{1}{2}\,t_{k^*}''(x_0)>0$. The proof for $A<0$ is similar and will be omitted. We adapt the approach of \cite{labate:detection,schober:detection} and substitute $v=x-x_0$ to get $\tilde{a}_{k^*}\mathrel{\mathop:}=a_{k^*}-x_0$, and for \cref{eq:L_k} we have
 		  \begin{equation*}
		G_{k^*}(\rho,\theta,n)=\int\limits_{\tilde{a}_{k^*}}^{\tilde{a}_{k^*+1}}\mathrm{e}^{-\mathrm{i}2^j\rho\boldsymbol{\Theta}^{\mathrm{T}}(\theta)\left(t_{k^*}(x_0)+Bv+Av^2+\mathcal{O}(v^3),v+x_0\right)^{\mathrm{T}}}\,p_{\theta}^n(\boldsymbol{\gamma}(v+x_0))\,\boldsymbol{\Theta}^{\mathrm{T}}(\theta)\,\boldsymbol{\beta}(v+x_0)\,\mathrm{d}v.
 		  \end{equation*}
 		 The result of \cite[Lemma 1]{schober:detection} implies
    		\begin{equation}\label{eq:supp_2hochj_psi}
    			\mathrm{supp}\,\Psi^{(\mathfrak{h})}_{j,\ell}(2^j\rho\,\boldsymbol{\Theta}(\theta))\subset\left\lbrace(\rho,\theta)\in \mathbb{R}\times\left[-\frac{\pi}{2},\frac{\pi}{2}\right]:\frac{1}{3}<\abs{\rho}< 2,\,\theta_{j,\ell-2}^{(\mathfrak{h})}<\theta_t<\theta_{j,\ell+2}^{(\mathfrak{h})}\right\rbrace.
    		\end{equation} 
			By assumption, for the directional derivatives on the boundary we have
			 		\begin{equation*}
			 			p_{\theta}^m(\boldsymbol{\gamma}(x))\begin{cases}
			 			=0, &\text{if } 0\leq m<n,\\
			 			\neq 0, &\text{if } m=n,
			 			\end{cases}
			 		\end{equation*}
					for $\theta\in \left( \theta_{j,\ell-2}^{(\mathfrak{h})},\theta_{j,\ell+2}^{(\mathfrak{h})} \right)$, why $I_{k^*}(j,\ell,\mathbf{y},m)=0$ for $0\leq m<n$ and $p_\theta^0(\boldsymbol{\gamma}(x))=p(\boldsymbol{\gamma}(x))\neq 0$ for $n=0$. As the proof will show, we only need to consider the integral $I_{k^*}(j,\ell,\mathbf{y},n)$ since the integrals $I_{k^*}(j,\ell,\mathbf{y},m)$ for $n<m_1\leq u$ decay faster.\\
					
 		For the integral $I_{k^*,1}^{(\mathfrak{h})}(j,\ell,\mathbf{y},m)$ we get
 		  \begin{equation*}
			I_1=2^{-j(n-1)}\int\limits_{\frac{1}{3}}^{2}\int\limits_{\theta_{j,\ell-2}^{(\mathfrak{h})}}^{\theta_{j,\ell+2}^{(\mathfrak{h})}}\int\limits_{\tilde{a}_{k^*}}^{\tilde{a}_{k^*+1}}\Psi_{j,\ell}^{(\mathfrak{h})}\left( 2^j\rho\,\boldsymbol{\Theta}(\theta) \right)\,\rho^{-n}\,\mathrm{e}^{\mathrm{i} 2^j\rho R(v,\theta)}\,\varphi(v,\theta)\,\mathrm{d}v\,\mathrm{d}\theta\, \mathrm{d}\rho,
 			\end{equation*}
    		where $\Lambda\mathrel{\mathop:}=2^j\rho$, $\varphi(v,\theta)\mathrel{\mathop:}=p_{\theta}^n(\boldsymbol{\gamma}(v+x_0))\,\boldsymbol{\Theta}^{\mathrm{T}}(\theta)\,\boldsymbol{\beta}(v+x_0)$ and 
 	      		\begin{align}
 					R(v,\theta)&\mathrel{\mathop:}=-\boldsymbol{\Theta}^{\mathrm{T}}(\theta) \bigl(Av^2+Bv+t_{k^*}(x_0)+\mathcal{O}(v^3)-2\pi y_1,v+x_0-2\pi y_2\bigr)^{\mathrm{T}}\notag\\
 					&=-\cos\theta\bigl( Av^2+(B+\tan\theta)v+t_{k^*}(x_0)+\mathcal{O}(v^3)-2\pi y_1+(x_0-2\pi y_2)\tan\theta\bigr)\notag\\
 					&=-\cos\theta\left(A\biggl(v+\frac{B+\tan\theta}{2A}\biggr)^2 +\widetilde{C}-2\pi y_1 - 
 	      		 \frac{(B+\tan\theta)^2}{4A}\right)\label{eq:R_A0}.
 	      		\end{align}
 				In the last line $\widetilde{C}\mathrel{\mathop:}=t_{k^*}(x_0)+(x_0-2\pi y_2)\tan\theta+r(v)$ and since $\abs{v}<\varepsilon=\varepsilon_0\,2^{-j/2}$ we have $\abs{r(v)}< C_1\,\varepsilon^3=C_2(\varepsilon_0)\,2^{-3j/2}$. It follows that
 				\begin{equation*}
 					\frac{\partial R}{\partial v}(v,\theta)=-2A\cos\theta\biggl(v+\frac{B+\tan\theta}{2A}\biggr)^2=0
 				\end{equation*}
 				if $v_{\theta}=-\frac{B+\tan\theta}{2A}$ and we introduce $\phi(v,\theta)\mathrel{\mathop:}=R(v,\theta)-R(v_{\theta},\theta)$ which gives
 				\begin{equation*}
 					\phi(v_{\theta},\theta)=\frac{\partial \phi}{\partial v}(v_{\theta},\theta)=0,\qquad\qquad\frac{\partial \phi^2}{\partial v^2}(v_{\theta},\theta)=\frac{\partial R^2}{\partial v^2}(v_{\theta},\theta)=-2A\cos\theta\neq0,
 				\end{equation*}
 				since $\cos\theta>0$ for $\theta\in \left(\theta_{j,\ell-2}^{(\mathfrak{h})},\theta_{j,\ell+2}^{(\mathfrak{h})} \right)$. This allows us to write $I_1$ as
 		 \begin{equation}\label{eq:I1} I_1=2^{-j(n-1)}\int\limits_{\frac{1}{3}}^{2}\int\limits_{\theta_{j,\ell-2}^{(\mathfrak{h})}}^{\theta_{j,\ell+2}^{(\mathfrak{h})}}\Psi_{j,\ell}^{(\mathfrak{h})}\left( 2^j\rho\,\boldsymbol{\Theta}(\theta) \right)\,\rho^{-n}\,\mathrm{e}^{\mathrm{i} 2^j\rho R(v_{\theta},\theta)}\left(\int\limits_{\tilde{a}_{k^*}}^{\tilde{a}_{k^*+1}}\mathrm{e}^{\mathrm{i}\, \Lambda\, \phi(v,\theta)}\, \varphi(v,\theta)\,\mathrm{d}v\right)\mathrm{d}\theta\,\mathrm{d}\rho.
 		 \end{equation}
  		We use \cite[Lemma 13]{schober:detection}, called method of stationary phase, for the inner integral to get the estimate
		  		 \begin{equation}\label{eq:I11}
		  		\int\limits_{\tilde{a}_{k^*}}^{\tilde{a}_{k^*+1}}\mathrm{e}^{\mathrm{i}\, \Lambda\, \phi(v,\theta)}\,\varphi(v,\theta)\,\mathrm{d}v=C\,\sqrt{\pi\mathrm{i}}\, (2^j\rho\,\abs{A\,\cos\theta})^{-\frac{1}{2}}\,\varphi(v_{\theta},\theta)+ r_2(j),
		  		 \end{equation}
		 		where $\abs{r_2(j)}\leq C_2\,2^{-j}$. As remarked in \cite[p. 115]{labate:detection} the constant $C_2>0$ is independent of $\theta$, $\rho$, $j$, $\ell$ and $\mathbf{y}$. Using \cref{eq:I11}, we further split up the integral \cref{eq:I1} in $I_1=I_{11}+I_{12}$ with
		 		 \begin{align*}
		&I_{11}=C\,2^{-j(n-1/2)}\sqrt{\frac{\pi\mathrm{i}}{A}}\int\limits_{\frac{1}{3}}^{2}\int\limits_{\theta_{j,\ell-2}^{(\mathfrak{h})}}^{\theta_{j,\ell+2}^{(\mathfrak{h})}}\Psi_{j,\ell}^{(\mathfrak{h})}\left( 2^j\rho\,\boldsymbol{\Theta}(\theta) \right)\,\rho^{-(n+1/2)}\,\mathrm{e}^{\mathrm{i} 2^j\rho R(v_{\theta},\theta)}\,\abs{\cos\theta}^{-1/2}\,\varphi(v_{\theta},\theta)\,\mathrm{d}\theta\,\mathrm{d}\rho,\\
		&I_{12}=C_2\,2^{-jn}\int\limits_{\frac{1}{3}}^{2}\int\limits_{\theta_{j,\ell-2}^{(\mathfrak{h})}}^{\theta_{j,\ell+2}^{(\mathfrak{h})}}\Psi_{j,\ell}^{(\mathfrak{h})}\left( 2^j\rho\,\boldsymbol{\Theta}(\theta) \right)\,\rho^{-n}\,\mathrm{e}^{\mathrm{i} 2^j\rho R(v_{\theta},\theta)}\,\mathrm{d}\theta\,\mathrm{d}\rho.
		 		 \end{align*}
		 			With the substitution $t=2^{j/2}\tan\theta-\ell$ we have $\mathrm{d}\theta=2^{-j/2}\cos^2{\theta_t}\,\mathrm{d}t$ where $\theta_t\mathrel{\mathop:}=\arctan((\ell+t)\,2^{-j/2})=\theta_{j,\ell+t}^{(\mathfrak{h})}$. For the function $R(v_{\theta},\theta)$ from \cref{eq:R_A0} this leads to 
		 				\begin{equation*}
		 					2^jR(v_{\theta_t},t)=\cos\theta_t\left(\frac{(2^{j/2}B+\ell+t)^2}{4A}-2^j(\widetilde{C}-2\pi y_1) \right)=\cos\theta_t\left(\frac{(p+t)^2}{4A} +D \right),
		 				\end{equation*}
		 				where $p\mathrel{\mathop:}=2^{j/2}B+\ell$ and $D\mathrel{\mathop:}=2^j(2\pi y_1-\widetilde{C})$. By assumption, we have $\mathbf{x}_0\in U_{\varepsilon}(\mathbf{y})$ such that $\abs{p}\leq \frac{1}{4}$ and $\abs{D}\leq \frac{3\pi}{4}$. \\
				
		From \cref{eq:supp_2hochj_psi}, it follows that $I_{11}=I_{12}=0$ for $\abs{t}>2$ and we get
		    		 \begin{align}
		&I_{11}=C\,2^{-jn}\sqrt{\frac{\pi\mathrm{i}}{A}}\int\limits_{\frac{1}{3}}^{2}\int\limits_{-2}^{2}\Psi_{j,\ell}^{(\mathfrak{h})}\left( 2^j\rho\,\boldsymbol{\Theta}(\theta_t) \right)\,\rho^{-(n+1/2)}\,\mathrm{e}^{\mathrm{i}\rho\cos\theta_t\left(\frac{(p+t)^2}{4A} +D \right)}\,\abs{\cos{\theta_t}}^{3/2}\,\varphi(v_{\theta_t},\theta_t)\,\mathrm{d}t\,\mathrm{d}\rho,\label{I_11_temp}\\\notag
		&I_{12}=C_2\,2^{-j(n+1/2)}\int\limits_{\frac{1}{3}}^{2}\int\limits_{-2}^{2}\Psi_{j,\ell}^{(\mathfrak{h})}\left( 2^j\rho\,\boldsymbol{\Theta}(\theta_t) \right)\,\rho^{-n}\,\mathrm{e}^{\mathrm{i}\rho\cos\theta_t\left(\frac{(p+t)^2}{4A} +D \right)}\,\abs{\cos{\theta_t}}^{2}\,\mathrm{d}t\,\mathrm{d}\rho.
		    		 \end{align}
		 		 A direct estimate with the triangle inequality leads to $\abs{I_{12}}\leq C_3\,2^{-j(n+1/2)}$ and we can omit this term in the following.
		 
		 		From the definition of $\theta_{j,\ell}^{(\mathfrak{h})}$ in \cref{eq:theta_jl} we get
		 		\begin{equation*}
		 			\cos\theta_t=\cos \left( \arctan \left( 2^{-j/2}(\ell+t)\right) \right)=\left( 1+\left( 2^{-j/2}(\ell+t) \right)^2 \right)^{-1/2}.
		 		\end{equation*} 
		 		The right-hand side is a function
		 		\begin{equation*}
		 			w(x)=\left( 1+\left( 2^{-j/2}\ell+x \right)^2 \right)^{-1/2}
		 		\end{equation*}
		 		evaluated in $x=2^{-j/2}\,t$. Using the Taylor approximation of order zero, we obtain
		 		\begin{equation*}
		 			h(2^{-j/2}t)=\left( 1+\left( 2^{-j/2}\ell\right)^2 \right)^{-1/2}+w'(\xi)\,2^{-j/2}t
		 		\end{equation*} 
		 		with $\abs{\xi}\leq \abs{2^{-j/2}t}\leq 2^{-j/2+1}$ since $\abs{t}\leq 2$. This leads to
		 		\begin{equation*}
		 			\abs{2^{-j/2}t\,w'(\xi)}\leq2^{-j/2+1}\frac{\abs{2^{-j/2}+\xi}}{\left( 1+\left( 2^{-j/2}\ell+\xi\right)^2\right)^{3/2}}\leq \frac{3\cdot2^{-j+1}}{\left( 1+\left( 2^{-j/2}\ell+\xi\right)^2\right)^{3/2}},
		 		\end{equation*}
		 		and we write $\cos\theta_t=\mu_{j,\ell}+r_3(j)$ with $r_3(j)=\mathcal{O}\left(2^{-j}\right)$ and $\mu_{j,\ell}\mathrel{\mathop:}=(1+(2^{-j/2}\ell)^2)^{-1/2}$ fulfills $2^{-1/2}\leq\abs{\mu_{j,\ell}}\leq 1$. Similar to the previous case, we write $\sin\theta_t=2^{-j/2}\,\ell\,\mu_{j,\ell}+r_4(j)$ with $r_4(j)=\mathcal{O}\left(2^{-j}\right)$. We omit the additive term with fast decay and replace $\cos\theta_t$ by $\mu_{j,\ell}$ and $\sin\theta_t$ by $2^{-j/2}\,\ell\,\mu_{j,\ell}$.
		
		 		 The curve $\boldsymbol{\gamma}(x)$ is parametrized by $(t_{k^*}(x),x)^{\mathrm{T}}$ for $x\in[a_{k^*},a_{k^*+1})$ leading to
		 		\begin{equation*} 
					\boldsymbol{\beta}(v_{\theta_t}+x_0)=\mathbf{n}(\boldsymbol{\gamma}(v_{\theta_t}+x_0))\abs{\boldsymbol{\gamma}'(v_{\theta_t}+x_0)}_2=\left( -1,t'(v_{\theta_t}+x_0) \right)^{\mathrm{T}}\sqrt{t'(v_{\theta_t}+x_0)^2+1}
		 		\end{equation*}
		 		and therefore
		 		\begin{equation*}
		 			\boldsymbol{\Theta}^{\mathrm{T}}(\theta_t)\,\boldsymbol{\beta}(v_{\theta_t}+x_0)=\left( \mu_{j,\ell}\left( 2^{-j/2}\,\ell\,t'(v_{\theta_t}+x_0)-1 \right) \right)\sqrt{t'(v_{\theta_t}+x_0)^2+1}.
		 		\end{equation*}
		 	 Moreover, by the assumption on the directional derivative of order $n$ there is $\widetilde{q}$ such that 
		 		\begin{equation*}
		 			\abs{p_{\theta_t}^n(\boldsymbol{\gamma}(v_{\theta_t}+x_0))-p_{\theta_t}^n(\boldsymbol{\gamma}(\widetilde{q})}\leq C\,2^{-j/2}
		 		\end{equation*}
		 		and $p_{\theta_t}^n(\boldsymbol{\gamma}(\widetilde{q})\neq 0$. We replace $\varphi(v_{\theta_t},\theta_t)=p_{\theta_t}^n(\boldsymbol{\gamma}(v_{\theta_t}+x_0))\,\boldsymbol{\Theta}^{\mathrm{T}}(\theta_t)\,\boldsymbol{\beta}(v_{\theta_t}+x_0)$ in \cref{I_11_temp} by a constant and write
		 		\begin{equation*}
		I_{11}=C_3\,2^{-jn}\,\mu_{j,\ell}^{3/2}\,\sqrt{\frac{\mathrm{i}}{A}}\int\limits_{\frac{1}{3}}^{2}\int\limits_{-2}^{2}\Psi_{j,\ell}^{(\mathfrak{h})}\left( 2^j\rho\,\boldsymbol{\Theta}(\theta_t) \right)\,\rho^{-(n+1/2)}\,\mathrm{e}^{\mathrm{i}\rho\,\mu_{j,\ell}\left(\frac{(p+t)^2}{4A} +D \right)}\,\mathrm{d}t\,\mathrm{d}\rho.
		 		\end{equation*}
		 		Next, we write $\lambda=\rho\,\mu_{j,\ell}$ which gives
		 		\begin{equation*}
		\Psi^{(\mathfrak{h})}_{j,\ell}(2^j\rho\,\boldsymbol{\Theta}(\theta_t))=\widetilde{g}(\rho\cos\theta_t)\,g\left(\rho\cos\theta_t(2^{j/2}\tan\theta_t-\ell)\right)=\widetilde{g}(\lambda)\,g\left(t\,\lambda\right).
		 		\end{equation*} 
		 		From here we can follow the steps from \cite[p. 36]{schober:detection} with the obvious changes in our case and write
		   		 \begin{equation}\label{eq:H_I11} 		I_{11}=C_3\,2^{-jn}\,\mu_{j,\ell}^{n+1}\,\sqrt{\frac{\mathrm{i}}{A}}\int\limits_{\frac{1}{3}}^{2}\widetilde{g}(\lambda)\,\mathrm{e}^{\mathrm{i}\lambda D}\,\lambda^{-(n+1/2)}\,H(\lambda,p,A)\,\mathrm{d}\lambda,
		   		 \end{equation}
		  		 where
		  		 \begin{equation*} 		
					 H(\lambda,p,A)\mathrel{\mathop:}=\sqrt{\frac{A}{\lambda}}\Bigl(a(\lambda,p,A)+\mathrm{i}\,b(\lambda,p,A)\Bigr)
		  		 \end{equation*}
		 		and
		 		 \begin{align*}
		 		a(\lambda,p,A)&\mathrel{\mathop:}=\int\limits_{0}^{\infty}\left(g\left(2\sqrt{A\lambda\,v}+ p\,\lambda\right)+g\left(2\sqrt{A\lambda\,v}- p\,\lambda\right)\right)\,\frac{\cos{v}}{\sqrt{v}}\,\mathrm{d}v,\\
				b(\lambda,p,A)&\mathrel{\mathop:}=\int\limits_{0}^{\infty}\left(g\left(2\sqrt{A\lambda\,v}+ p\,\lambda\right)+g\left(2\sqrt{A\lambda\,v}- p\,\lambda\right)\right)\,\frac{\sin{v}}{\sqrt{v}}\,\mathrm{d}v.
		 		 \end{align*}
				With the positive solution $\sqrt{\mathrm{i}}=\frac{1+\mathrm{i}}{\sqrt{2}}$ we can write \cref{eq:H_I11} as $I_{11}=\mathrm{Re}(I_{11})+\mathrm{i}\,\mathrm{Im}(I_{11})$ with
		    		 \begin{align*}  
		 			 \mathrm{Im}(I_{11})&=C_4(n,A)\,2^{-jn}\int\limits_{\frac{1}{3}}^{\frac{4}{3}}\widetilde{g}(\lambda)\,\lambda^{-(n+1)}\,\Bigl( \Bigl[a(\lambda,p,A)+b(\lambda,p,A)\Bigr]\cos(D\lambda)\\
		 			&\qquad+\Bigl[a(\lambda,p,A)-b(\lambda,p,A)\Bigr]\sin(D\lambda) \Bigr)\mathrm{d}\lambda.
		    		 \end{align*}
		    		We use the relation $I_{2}=-\overline{I_{1}}$ from \cref{eq:I1_I2} and repeat the previous steps for the integral $I_{2}$ instead of $I_{1}$ to get $I_{21}=\mathrm{Re}(I_{21})+\mathrm{i}\,\mathrm{Im}(I_{21})$ with
		    		 \begin{align*}   \mathrm{Im}(I_{21})&=C_5(n,A)\,2^{-jn}\int\limits_{\frac{1}{3}}^{\frac{4}{3}}\widetilde{g}(\lambda)\,\lambda^{-(n+1)}\,\Bigl(\Bigl[a(\lambda,p,A)+b(\lambda,p,A)\Bigr]\sin(D\lambda)\\
		 			 &\qquad-\Bigl[a(\lambda,p,A)-b(\lambda,p,A)\Bigr]\cos(D\lambda) \Bigr)\mathrm{d}\lambda.
		    		 \end{align*}
		 		 The integrals $\mathrm{Im}(I_{11})$ and $\mathrm{Im}(I_{21})$ are up to the factor $\lambda^{-(n+1)}$ identical to $P_1(D,p,A)$ and $P_2(D,p,A)$ from \cite[Lemma 16]{schober:detection} which can similarly be shown to hold true in our case. The assumptions $\abs{p}\leq \frac{1}{2}$ and $\abs{D}\leq \frac{3\pi}{4}$ of that lemma are also fulfilled.
		    		From \cref{eq:I1_I2} we obtain 
		 		\begin{equation*}
		 			\bigl\lvert I_{k^*}(j,\ell,\mathbf{y},m)\bigr\rvert =2\bigl\lvert\mathrm{Im}(I_{11}+I_{12})\bigr\rvert=2\bigl\lvert\mathrm{Im}(I_{21}+I_{22})\bigr\rvert
		 		\end{equation*}
		 		and with the inverse triangle inequality and \cite[Lemma 16]{schober:detection} we have shown the lower bound \cref{eq:haupttheorem2_beh} and thus the estimate \cref{eq:lower_bound}. We have proven \cref{proof:hauptresultat2_0} in the case $A>0$. For $A=0$ we omit the proof since the arguments from \cite{schober:detection} in that case can merely be repeated with the obvious modifications. \\
				 
				In the last part of the proof we show the lower bound of \cref{thm:hauptresultat2} for cartoon-like functions with \cref{lem:P_L}. In order to do that, we consider functions $\mathfrak{f}_0\in \mathcal{E}^{u+1}(\tau)$ of the form $\mathfrak{f}_0=f_0+f\,\chi_T=f_0+\mathfrak{f}$ with $\mathfrak{f}=f\,\chi_T$ and $f_0,f\in C_0^{u+1}(\mathbb{R}^2)$. In \cref{eq:f_0} we show that
		    	 \begin{equation*} 
		    	 	\abs{\left\langle f_0^{2\pi},\psi^{(\mathfrak{i})}_{j,\ell,\mathbf{y}}\right\rangle_2}\leq C_3(u,q)\,2^{-j(q+u+1/2)}.
		    	 \end{equation*}
		We use the inverse triangle inequality, \cref{proof:hauptresultat2_0}, \cref{lem:P_L} and the assumption $u>4(n+1)$ to get
		 			 \begin{align*}
		 			 	\abs{\left\langle \mathfrak{f}_0^{2\pi},\psi_{j,\ell,\mathbf{y}}^{(\mathfrak{i})}\right\rangle_2}&\geq\abs{\left\langle P_{u,f,\mathbf{y}}^{2\pi},\psi_{j,\ell,\mathbf{y}}^{(\mathfrak{i})} \right\rangle_2}-\abs{\left\langle \mathfrak{f}^{2\pi}-P_{u,f,\mathbf{y}}^{2\pi},\psi_{j,\ell,\mathbf{y}}^{(\mathfrak{i})} \right\rangle_2}-\abs{\left\langle f_0^{2\pi},\psi^{(\mathfrak{i})}_{j,\ell,\mathbf{y}}\right\rangle_2}\\
		 				&\geq C_1(u,n,q,\varepsilon_0,T)\,2^{-j(3/4+n)}-C_2(\mathfrak{f},q)\,2^{-j(u-1)/4}-C_3(u,q)\,2^{-j(q+u+1/2)}\\
		 				&\geq C_4 (u,n,q,\varepsilon_0,T)\,2^{-j(3/4+n)}
		 			 \end{align*}
		 			 and \cref{thm:hauptresultat2} is proven. \qed

% Section proof_of_theorem_3_2 (end) 

% Section conclusion_and_outlook (end)

\end{document}